\documentclass[11pt,a4paper,reqno]{amsart}

\usepackage[a4paper,left=2.0cm,right=2.0cm,top=2.5cm,bottom=2.5cm]{geometry}

\usepackage{times} %
\usepackage{amsmath} %
\usepackage{amssymb}  %
\usepackage{amsthm}
\usepackage{latexsym}
\usepackage{amsfonts,bbm}
\usepackage{xcolor}
\usepackage{mathtools}
\usepackage{enumerate}
\usepackage{cite}
\usepackage[foot]{amsaddr}
\usepackage{algorithm,algorithmic}
\usepackage[a4paper,left=2.0cm,right=2.0cm,top=2.5cm,bottom=2.5cm]{geometry}
\usepackage{multicol}

\usepackage{tikz}
\usetikzlibrary{calc,positioning, arrows.meta, shapes.geometric}
\usepackage{xargs}
\usepackage[hidelinks]{hyperref}
\usepackage{textcomp}

\newtheorem{theorem}{Theorem}[section]
\newtheorem{corollary}[theorem]{Corollary}

\newtheorem{proposition}[theorem]{Proposition}

\newtheorem{remark}[theorem]{Remark}

\newtheorem{assumption}[theorem]{Assumption}

\usepackage{xcolor}
\usepackage[normalem]{ulem}
\usepackage{mathtools}
\usepackage{todonotes}
\usepackage{bm}
\usepackage{appendix}
\usepackage{url}

\usepackage{amsmath,amssymb}
\usepackage{tikz}
\usepackage{pgfplots}
\usetikzlibrary{fit}
\usepgfplotslibrary{groupplots}

\usepackage{enumitem}
\usepackage{cleveref}

\newcommand{\calU}{\mathcal{U}}
\newcommand{\calY}{\mathcal{Y}}
\newcommand{\calX}{\mathcal{X}}

\newcommand{\calA}{\mathcal{A}}
\newcommand{\calB}{\mathcal{B}}
\newcommand{\calC}{\mathcal{C}}
\newcommand{\calQ}{\mathcal{Q}}
\newcommand{\calK}{\mathcal{K}}
\newcommand{\calP}{\mathcal{P}}
\newcommand{\calM}{\mathcal{M}}
\newcommand{\calW}{\mathcal{W}}

\newcommand{\calR}{\mathcal{R}}
\newcommand{\calZ}{\mathcal{Z}}

\newcommand{\tcalA}{\widetilde{\mathcal{A}}}

\newcommand{\Ry}{\mathcal{R}_y}
\newcommand{\Ru}{\mathcal{R}_u}
\newcommand{\R}{\mathbb{R}}

\newcommand{\dom}[1]{\operatorname{dom}(#1)}
\newcommand{\ran}[1]{\operatorname{ran}(#1)}
\DeclareMathOperator{\re}{Re}

\newcommand{\blk}{\color{black}}

\allowdisplaybreaks

\newenvironment{smallpmatrix}%
{\left(\begin{smallmatrix}}%
{\end{smallmatrix}\right)}%

\newenvironment{smallbmatrix}%
{\left[\begin{smallmatrix}}%
{\end{smallmatrix}\right]}%

\pgfplotsset{
    dl1/.style={
        legend image code/.code={
            \draw[blue,dashed]
                (0cm,0cm) -- (0.5cm,0cm);
            \draw[blue,mark=o]
                plot coordinates {(0.25cm,0cm)};
            \node at (0.6cm,0cm) {\small /};
            \draw[red!70!black,dashed]
                (0.7cm,0cm) -- (1.2cm,0cm);
            \draw[red!70!black,mark=square*]
                plot coordinates {(0.95cm,0cm)};
        }
    }
}

\pgfplotsset{
    dl2/.style={
        legend image code/.code={
            \draw[blue,solid]
                (0cm,0cm) -- (0.5cm,0cm);
            \draw[blue,mark=o]
                plot coordinates {(0.25cm,0cm)};
            \node at (0.6cm,0cm) {\small /};
            \draw[red!70!black,solid]
                (0.7cm,0cm) -- (1.2cm,0cm);
            \draw[red!70!black,mark=square*]
                plot coordinates {(0.95cm,0cm)};
        }
    }
}

\title{Stabilize-then-optimize: Feedback transformations as preconditioners in optimal control}

\author{Till Preuster$^1$}\address{$^1$Faculty of Mathematics, Chemnitz University of Technology, Germany\\ Mail: \textsc{\{till.preuster,manuel.schaller,martin.stoll\}@math.tu-chemnitz.de}}
\author{Anton Schiela$^2$}\address{$^2$Department of Mathematics, University of Bayreuth, Germany\\ Mail: \textsc{anton.schiela@uni-bayreuth.de}}
\author{Manuel Schaller$^1$}
\author{Martin Stoll$^1$}

\thanks{Till Preuster and Manuel Schaller acknowledge funding by the Deutsche Forschungsgemeinschaft (DFG, German Research Foundation) – Project-ID 531152215 – CRC 1701. Manuel Schaller is further supported by the DFG under Project-ID 519323897. }

\begin{document}

\begin{abstract}
Many numerical algorithms for optimal control leverage an elimination of the state via the control-to-state map such as condensed approaches or preconditioned conjugate gradient methods for the optimality system. As such, the norm of the control-to-state map directly enters the convergence estimates for these methods, e.g., via the condition number of the associated linear system. In this work we show that using feedback transformations one may reformulate the optimal control problem to decrease the norm of the (feedbacked) control-to-state map, leading to a drastic improvement of the involved condition numbers. We illustrate the abstract approach for ordinary and partial differential equations such as parabolic, hyperbolic or elliptic equations. For each of these problem classes we provide a constructive method to improve solution operator norms via feedbacks. Further, we showcase the efficacy of the method by means of various numerical examples with elliptic, parabolic and hyperbolic partial differential equations. 
\end{abstract}

\maketitle
\smallskip

\section{Introduction}
Optimal control of differential equations is a well-established field of mathematical analysis, numerical methods and scientific computing. Problems subject to ordinary or partial differential equations (ODEs or PDEs) have been thoroughly analyzed and numerically approached~\cite{Tro10,hinze2008optimization,lions1971optimal}. The fundamental structure in these systems is that pairs of optimization variables $(x,u)$ satisfying the PDE may be characterized via a control-to-state map. This structure is useful analytically but may also be used to derive efficient numerical algorithms as the control-to-state map immediately provides access to a projection onto the feasible set. %

Many PDE-constrained optimal control problems may be viewed as a realization of the abstract linear-quadratic optimal control problem (OCP)
\begin{equation}\label{eq:OCP}
    \min_{(x,u)\in \calX%
    \times \calU} \frac12\|\mathcal{C}x-y_\mathrm{ref}\|_\calY^2 + \frac\alpha2 \|u\|_\calU^2 \quad \mathrm{s.t.} \quad \mathcal{A}x-\mathcal{B}u = f,
\end{equation}
where
\begin{enumerate}[label=(\roman*)]
    \item $\calX$ is a Banach space and $\calY,\calU,\calP$ are Hilbert spaces,
    \item $\calA\in L(\calX,\calP^*)$ is continuously invertible and encodes some ODE or PDE,
    \item $\calB \in L(\calU,\calP^*)$ is an input operator and  $\calC\in L(\calX,\calY)$ is an observation operator,
    \item[(iv)] $y_\mathrm{ref} \in \calY$ is a reference output, $f \in \calP^*$ is a forcing term, and $\alpha > 0$ is a regularization parameter.
\end{enumerate}
Here, the constraint could correspond either to a static problem that is, an elliptic PDE for which $\calA$ is usually a second-order (elliptic) operator or a time-dependent problem such as a parabolic or hyperbolic partial differential equation, where the constraint operator $\calA$ involves space and time derivatives. Such linear-quadratic problems occur, e.g., in in each step of an SQP-method applied to nonlinear problems.

One of the simplest approaches to solve \eqref{eq:OCP} is to eliminate the PDE-constraint using invertibility of $\mathcal{A}$, that is, leveraging the control-to-state map $u \mapsto \calA^{-1}\calB u$. Hence, setting $x=\mathcal{A}^{-1}(\calB u + f)$ one may eliminate the state from the optimal control problem and consider the reduced or condensed problem. Then, standard approaches such as gradient methods may be applied \cite{hinze2008optimization,Tro10}. Alternatively, one may keep the constraint and introduce a Lagrange multiplier and derive suitable first-order optimality conditions in function space. Then, these could be approached by linear or nonlinear solvers, leading to an algorithm in state, control and adjoint variables such as conjugate gradient methods endowed with suitably designed constraint preconditioners \cite{keller2000constraint,rees2010optimal,schiela2014operator}. %

 In both of these approaches, however, the condition number of the systems to be solved strongly depends on the norm of the solution operator $\|\mathcal{A}^{-1}\|_{\calP^*\to \calX}$, where $\| \cdot \|_{\calP^* \to \calX}$ denotes the standard operator norm in $L(\calP^*,\calX)$. This operator norm, however, can be prohibitively large, e.g., when considering unstable time-dependent problems on large time horizons or elliptic problems with small ellipticity constant. While usually this is treated as given by the problem, we show in this work how to shape the constraint operator $\mathcal{A}$ by an equivalence transformation such that the transformed problem is subject to an operator $\widetilde{\mathcal{A}} = \calA + \calB\calK$, where $\calK \in L(\calX,\calU)$ is a feedback operator leading to the property $\|{\widetilde{\calA}}^{-1}\|_{\calP^*\to \calX} \ll \|{\calA}^{-1}\|_{\calP^*\to \calX}$. Feedback transformations and feedback stabilization are classical and well-understood tools in mathematical systems theory~\cite{hinrichsen2026mathematical}. While they are typically studied from a control-theoretic perspective, we show that they can also be leveraged in optimal control to drastically improve the conditioning and performance of numerical solvers.

 The central tools are feedback stabilizations that are borrowed from mathematical systems theory where they are usually used to stabilize dynamic control systems; here, we showcase how they may also be leveraged for reformulations in optimal control in view of efficiency and robustness of numerical methods.

To illustrate the guiding idea of this work,  we introduce a \emph{state feedback operator} $\mathcal{K}:\calX \to \calU$ and substitute 
\begin{equation}\label{eq:feedback}
    u=\mathcal{K}x+v,
\end{equation}
where $x\in \calX$ is the state and $v\in \mathcal{U}$ is a new control variable. In this new control variable, the OCP \eqref{eq:OCP} reads
\begin{equation}\label{eq:OCP2}
    \min_{(x,v)\in \calX \times \calU} \frac12\|\mathcal{C}x-y_\mathrm{ref}\|_\calY^2 + \frac\alpha2 \|\mathcal{K}x+v\|_\calU^2 \quad \mathrm{s.t.} \quad (\mathcal{A}-\mathcal{B}\mathcal{K})x - \mathcal{B}v = f.
\end{equation}
This problem is clearly equivalent to the problem~\eqref{eq:OCP} as the transformation \eqref{eq:feedback} can easily be inverted via $v = u - \calK x$, hence it defines a bijection (independently of $x\in \calX$). However, as illustrated in Figure~\ref{fig:spoilerfig}, a suitably chosen feedback operator $\calK$ leads to strongly improved convergence behavior of iterative numerical optimal control solvers applied to \eqref{eq:OCP2} instead of \eqref{eq:OCP}.

In the following we abbreviate the norm of the stabilized control-to-observation map as follows:
\begin{equation}\label{eq:norm-c-to-obs}
\sigma_\calK := \|\calC(\calA-\calB\calK)^{-1}\calB\|_{\calU \to \calY}
\end{equation}
so that the unfeedbacked cases is recovered by $\sigma_0=\|\calC\calA^{-1}\calB\|_{\calU \to \calY}$ which relates to the original control-to-observation map of \eqref{eq:OCP}.

\begin{figure}[htb]
\centering
\scalebox{0.73}{
\begin{tikzpicture}
\node (a) at (0,0)
{
    \scalebox{0.9}{%

\begin{tikzpicture}

\begin{axis}[
    width=0.35\textwidth,
    height=0.3\textwidth,
    xlabel={CG iterations},
    ylabel={Relative error},
    xmin=5, xmax=50,
    xtick={5,10,15,20,25,30,35,40,50},
    ymode=log,
    ytickten={0,-2,-4,-6,-8,-10,-12,-14},
    grid=both,
    legend style={
            font=\footnotesize,
        at={(.4,.5)},
        anchor= west
    },
    thick,
    title={Discrete scalar problem}
]
\addlegendimage{dl1}
\addlegendentry{Stable}

\addlegendimage{dl2}
\addlegendentry{Unstable}

\addplot[mark=o, blue,densely dashed,    mark options={solid}] coordinates {
    (5, 0.04235891123124236)
    (10, 0.003048305976226194)
    (15, 0.00020441133789705587)
    (20, 4.822765685728518e-06)
    (25, 4.13927970851602e-08)
    (30, 1.1200188109752822e-09)
    (35, 2.3738785083640106e-11)
    (40, 3.26343209155649e-13)
    (50, 1.4792847974608376e-15)
};
\addplot[mark=square*, red!70!black,densely dashed,     mark options={solid}] coordinates {
    (5, 0.0015160728114093625)
    (10, 4.236997471622918e-07)
    (15, 1.1119691408312645e-10)
    (20, 1.4134736097013663e-14)
    (25, 0.0)
    (30, 0.0)
    (35, 0.0)
    (40, 0.0)
    (50, 0.0)
};

\addplot[mark=o, blue] coordinates {
    (5, 0.1863237056771352)
    (10, 0.05402498057324269)
    (15, 0.025819768885334955)
    (20, 0.007396670083771834)
    (25, 0.0036331373561403614)
    (30, 0.003097873548688374)
    (35, 0.0013303865047986766)
    (40, 0.0007100649639852083)
    (50, 0.000172465082642948)
};
\addplot[mark=square*, red!70!black] coordinates {
    (5, 0.0003424978231342578)
    (10, 5.569385315404921e-09)
    (15, 8.448807074076804e-14)
    (20, 0.0)
    (25, 0.0)
    (30, 0.0)
    (35, 0.0)
    (40, 0.0)
    (50, 0.0)
};

\end{axis}

\end{tikzpicture}
}
};
\node (b) at (5.5,0)
{
    \scalebox{0.9}{%

\begin{tikzpicture}

\begin{axis}[
    width=0.35\textwidth,
    height=0.3\textwidth,
    xlabel={CG iterations},
    xmin=10, xmax=30,
    xtick={10,20,30},
    ymode=log,
    ytickten={0,-2,-4,-6,-8,-10,-12,-14},
    grid=both,
    legend style={
        font=\footnotesize,
        at={(.2,1.1)},
        anchor=north west
    },
    thick,
    title={Elliptic problem}
]

\addlegendimage{dl1}
\addlegendentry{Large reac.\ }

\addlegendimage{dl2}
\addlegendentry{Small reac.\ }

\addplot[mark=o, blue,densely dashed,    mark options={solid}] coordinates {
    (10, 0.011733192233152737)
    (20, 3.274645659341223e-05)
    (30, 8.762544894739712e-08)
};
\addplot[mark=square*, red!70!black,densely dashed,    mark options={solid}] coordinates {
    (10, 0.0056951952936283535)
    (20, 5.507430247454323e-06)
    (30, 9.7472176661239e-09)
};

\addplot[mark=o, blue] coordinates {
    (10, 0.10641841984860713)
    (20, 0.004220039417801839)
    (30, 0.00023336971197845894)
    (40, 2.2300893678039672e-05)
};
\addplot[mark=square*, red!70!black] coordinates {
    (10, 0.008633129418152677)
    (20, 5.283177595836857e-06)
    (30, 1.1323333055598219e-08)
    (40, 2.309752168457207e-09)
};

\end{axis}

\end{tikzpicture}
}
};
\node (c) at (10.7,0)
{
    \scalebox{0.9}{%
\begin{tikzpicture}

\begin{axis}[
    width=0.35\textwidth,
    height=0.3\textwidth,
    xlabel={CG iterations},
    xmin=10, xmax=400,
    xtick={10,100,200,400},
    ymode=log,
    ytickten={0,-2,-4,-6,-8,-10,-12,-14},
    grid=both,
    legend style={
        at={(.5,1.1)},
        anchor= west
    },
    thick,
    title={Parabolic problem}
]

\addplot[mark=o, blue,densely dashed,    mark options={solid}] coordinates {
    (10, 0.8204462370764743)
    (20, 0.12301865727156086)
    (50, 0.0014286185413296363)
    (100, 1.1386718166156318e-08)
};
\addplot[mark=square*, red!70!black,densely dashed,    mark options={solid}] coordinates {
    (10, 0.30263638912738344)
    (20, 0.00884505808847306)
    (50, 1.0855608805835202e-07)
    (100, 7.924603262603529e-11)
};

\addplot[mark=o, blue] coordinates {
    (10, 0.9802554530941032)
    (20, 0.8286049752321615)
    (50, 0.3773840497283715)
    (100, 0.11251467131309723)
    (150, 0.0145298363665835)
    (200, 0.004216259455460214)
    (400, 1.943368372726511e-05)
};
\addplot[mark=square*, red!70!black] coordinates {
    (10, 0.5391080260343237)
    (20, 0.12628053278833976)
    (50, 0.0005799841527966209)
    (100, 1.2874032088309154e-08)
    (150, 3.8627187159246203e-13)
    (200, 0.0)
    (400, 0.0)
};

\end{axis}

\end{tikzpicture} %
}
};
\node (d) at (15,0)
{
    \scalebox{.9}{%
\begin{tikzpicture}

\begin{axis}[
    width=0.35\textwidth,
    height=0.3\textwidth,
    xlabel={CG iterations},
    xmin=10, xmax=200,
    xtick={20,50,100,200},
    ymode=log,
    ytickten={0,-2,-4,-6,-8,-10,-12,-14},
    grid=both,
        legend style={
        font=\footnotesize,
        at={(-.5,0.4)},
        anchor=north west
    },
    thick,
    title={Hyperbolic problem},
yticklabel pos=right
]

\addlegendimage{dl1}
\addlegendentry{Short horizon}

\addlegendimage{dl2}
\addlegendentry{Long horizon}

\addplot[mark=o, blue,densely dashed,    mark options={solid}] coordinates {
    (10, 0.6606344295840225)
    (20, 0.23851233397327157)
    (50, 0.01505220619328308)
    (100, 8.088854116884556e-05)
    (200, 4.616527754370169e-09)
};
\addplot[mark=square*, red!70!black,densely dashed,    mark options={solid}] coordinates {
    (10, 0.5779625783515272)
    (20, 0.20750711355316964)
    (50, 0.0055247715401029185)
    (100, 2.202974583132827e-05)
    (200, 4.119937050485917e-10)
};

\addplot[mark=o, blue] coordinates {
    (10, 0.9358957976889742)
    (20, 0.7450075952725995)
    (50, 0.39316701989374675)
    (100, 0.11787877169081454)
    (200, 0.016509783845688457)
};
\addplot[mark=square*, red!70!black] coordinates {
    (10, 0.8086783704646525)
    (20, 0.5874916599975385)
    (50, 0.1898510101126374)
    (100, 0.037881820964390245)
    (200, 0.001293915246537536)
};

\end{axis}

\end{tikzpicture} %
}
};
\node[
    draw,
    thick,
    rounded corners,
    inner sep=0.35cm,
    fit=(a)(b)(c)(d),
    label=above:{\Large Condensed gradient method}
] (upperbox) {};

\hspace{1.5cm}
\node (e) at (1,-6.7)
{
    \scalebox{0.9}{%
\begin{tikzpicture}

\begin{axis}[
    width=0.35\textwidth,
    height=0.3\textwidth,
    xlabel={CG iterations},
    ylabel={Relative error},
    xmin=1, xmax=80,
    xtick={10,20,30,40,50,60,70,80},
    ymode=log,
    ytickten={0,-2,-4,-6,-8,-10,-12,-14},
    grid=both,
      legend style={
            font=\footnotesize,
        at={(.5,1.1)},
        anchor=north west
    },    thick,
    title={Elliptic problem}
]

\addlegendimage{dl1}
\addlegendentry{Large reac.\ }

\addlegendimage{dl2}
\addlegendentry{Small reac.\ }

\addplot[mark=o, blue,densely dashed,    mark options={solid}] coordinates {
    (2, 0.04941758191711827)
    (4, 0.041436358256348296)
    (8, 0.02568596087456935)
    (15, 0.004946684162894679)
    (25, 0.0012084860081291373)
    (50, 7.129802433804703e-06)
    (75, 8.337032976813433e-08)
    (100, 1.7803515706513951e-09)
};
\addplot[mark=square*, red!70!black,densely dashed,    mark options={solid}] coordinates {
    (2, 0.04930440410061063)
    (4, 0.03722440928886701)
    (8, 0.013157310461967477)
    (15, 0.002412359042185279)
    (25, 0.0001442648043380194)
    (50, 2.1375563926056738e-07)
    (75, 1.7172529261798029e-09)
    (100, 1.0459318892870784e-11)
};

\addplot[mark=o, blue] coordinates {
    (2, 0.05238257281489796)
    (4, 0.04941993147227241)
    (8, 0.03723719612965626)
    (15, 0.010736828818260553)
    (25, 0.0050314560491186286)
    (50, 0.000535200602925313)
    (75, 5.8784638994884564e-05)
    (100, 1.1259554690676407e-05)
};
\addplot[mark=square*, red!70!black] coordinates {
    (2, 0.04930685303952718)
    (4, 0.040380852926308254)
    (8, 0.014763042615168049)
    (15, 0.0024128086422607973)
    (25, 0.00012162722684809979)
    (50, 2.2556879783423876e-07)
    (75, 9.453292612667716e-10)
    (100, 8.07171153385791e-12)
};

\end{axis}

\end{tikzpicture} %
}
};
\node (f) at (6.5,-6.7)
{
    \scalebox{0.9}{%
\begin{tikzpicture}

\begin{axis}[
    width=0.35\textwidth,
    height=0.3\textwidth,
    xlabel={CG iterations},
    xmin=5, xmax=100,
    xtick={10,30,50,70,100},
    ymode=log,
    ytickten={0,-2,-4,-6,-8,-10,-12,-14},
    grid=both,
    legend style={
        font=\footnotesize,
        at={(0,.4)},
        anchor=north west
    },
    thick,
    title={Parabolic problem}
]

\addlegendimage{dl1}
\addlegendentry{Stable}

\addlegendimage{dl2}
\addlegendentry{Unstable}

\addplot[mark=o, blue,densely dashed,    mark options={solid}] coordinates {
    (5, 0.8752619652898339)
    (10, 0.7823320151536354)
    (20, 0.3175891630951712)
    (30, 0.12515734706283974)
    (40, 0.06395796658386114)
    (50, 0.012722478465550543)
    (70, 0.001869615554663436)
    (100, 6.397775209645658e-05)
};
\addplot[mark=square*, red!70!black,densely dashed,    mark options={solid}] coordinates {
    (5, 0.7772406971596546)
    (10, 0.5557916930398292)
    (20, 0.14481341422679953)
    (30, 0.05392831574668583)
    (40, 0.009332051506731557)
    (50, 0.001753821274181866)
    (70, 8.197220252472123e-05)
    (100, 8.24252345299534e-07)
};

\addplot[mark=o, blue] coordinates {
    (5, 0.9995713035120756)
    (10, 0.9995712929177519)
    (20, 0.9953905894953937)
    (30, 0.995390592951449)
    (40, 0.9733121069420667)
    (50, 0.9733021753363749)
    (70, 0.9295989094586826)
    (100, 0.8984810405129386)
};
\addplot[mark=square*, red!70!black] coordinates {
    (5, 0.7948814185810744)
    (10, 0.5754407416104831)
    (20, 0.1551520811269918)
    (30, 0.062324726413066904)
    (40, 0.010565865682195941)
    (50, 0.0023274306220867416)
    (70, 0.00010917466373715135)
    (100, 1.3956089051444946e-05)
};

\end{axis}

\end{tikzpicture} %
}
};
\node (g) at (12.2,-6.7)
{
    \scalebox{0.9}{%
\begin{tikzpicture}

\begin{axis}[
    width=0.35\textwidth,
    height=0.3\textwidth,
    xlabel={CG iterations},
    xmin=5, xmax=100,
    xtick={10,30,50,70,100},
    ymode=log,
    ytickten={0,-2,-4,-6,-8,-10,-12,-14},
    grid=both,
    legend style={
            font=\footnotesize,
        at={(.5,1.1)},
        anchor=north west
    },
    thick,
    title={Hyperbolic problem}
]

\addlegendimage{dl1}
\addlegendentry{Short horizon}

\addlegendimage{dl2}
\addlegendentry{Long horizon}

\addplot[mark=o, blue,densely dashed,    mark options={solid}] coordinates {
    (5, 0.7233214322811498)
    (10, 0.46343415405795013)
    (20, 0.11359970835551332)
    (30, 0.021643844094331417)
    (40, 0.008460147300783163)
    (50, 0.0026601395416901395)
    (70, 0.00016416335573618586)
    (100, 2.3862182700344523e-06)
};
\addplot[mark=square*, red!70!black, densely dashed,    mark options={solid}] coordinates {
    (5, 0.39319084407354893)
    (10, 0.0789630512648929)
    (20, 0.009069307755050389)
    (30, 0.0010619157247466439)
    (40, 8.111755649120188e-05)
    (50, 1.3248166691609693e-05)
    (70, 1.5636280758097447e-07)
    (100, 2.22646260037352e-10)
};

\addplot[mark=o, blue] coordinates {
    (5, 0.6948124208693326)
    (10, 0.46266920379974535)
    (20, 0.23795308959141517)
    (30, 0.02912696681086819)
    (40, 0.014429448303085604)
    (50, 0.0053772673676302845)
    (70, 0.0006680314243765)
    (100, 1.939757326766026e-05)
};
\addplot[mark=square*, red!70!black] coordinates {
    (5, 0.3965342315572299)
    (10, 0.08048055639922759)
    (20, 0.009069689010035933)
    (30, 0.0014628224177524172)
    (40, 0.0001447487151987071)
    (50, 2.2687018771083166e-05)
    (70, 4.1129797404699084e-07)
    (100, 8.738023919440107e-10)
};

\end{axis}

\end{tikzpicture} %
}
};
\node[
    draw,
    thick,
    rounded corners,
    inner sep=0.35cm,
    fit=(e)(f)(g),
    label={[yshift=-0.2cm]below:{\Large Projected preconditioned conjugate gradient method}}
    ] (lowerbox) {};

\node at (7,-3.35)
{
     \Large \textcolor{blue}{\Large Original problem \eqref{eq:OCP}} and \textcolor{red!70!black}{feedback-reformulated problem \eqref{eq:OCP2}}
};

\end{tikzpicture}}

\caption{Efficacy of the suggested feedback transformation for iterative optimal control solvers.}
\label{fig:spoilerfig}

\end{figure}
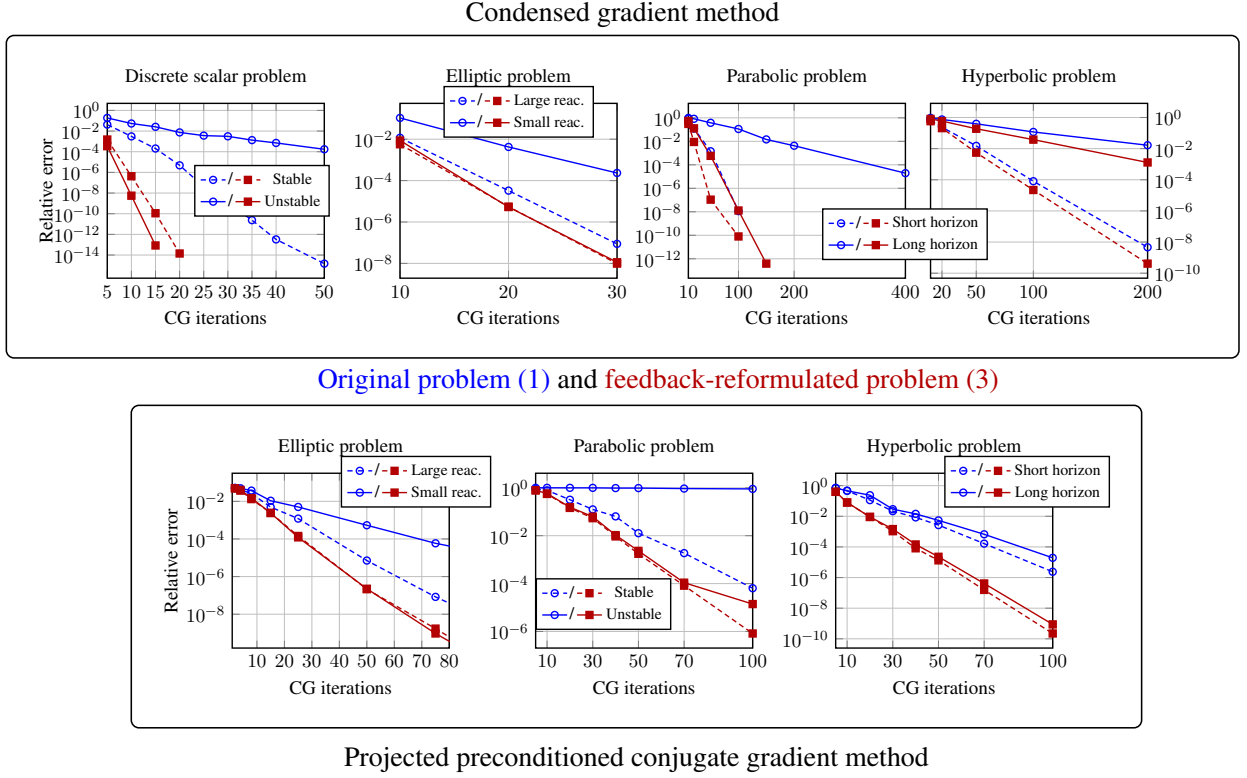

\textbf{Code availability}. The code for all numerical experiments presented in this work is available in the repository 
\begin{center}
   {\url{www.mytuc.org/mxxc}}
\end{center}

\section{Iterative methods for PDE-constrained optimal control}
In this part, we introduce two well-established methods in numerical optimal control that particularly benefit from our suggested feedback stabilization, namely a gradient method for the reduced problem as introduced in Subsection~\ref{subsec:gradintro} and a preconditioned conjugate gradient method for the optimality system as recalled in Subsection~\ref{subsec:ppcgintro}. 

In the context of iterative solvers for the optimality system corresponding to \Cref{eq:OCP2} it is crucial to design preconditioners that are tailored to the underlying problem formulation. The framework of \textit{operator preconditioning} for such systems (cf. \cite{mardal2011preconditioning,malek2014preconditioning,schiela2014operator}) requires bounding the \textit{energy} of the operator inner product by the preconditioner inner product.

While the most descriptive measure for convergence of Krylov methods is the distribution of the eigenvalues \cite{liesen2012krylov} it is often difficult to obtain robust bounds for the eigenvalues in context of optimal control, see \cite{pearson2012regularization} for the heat equation with full observation. A measure that is also very informative for the convergence behavior is the condition number \cite{saad2003iterative} giving rise to relative energy error estimates of the form
\begin{equation*}
\frac{\|e_k\|_{A}}{\|e_0 \|_{A}} \le 2 \left(\frac{\sqrt{\kappa(A)} - 1}{\sqrt{\kappa(A)} + 1} \right)^k 
\end{equation*}
 with $e_k$ being the error at step $k\in \mathbb{N}$ of the conjugate gradient method \cite{hestenes1952methods}. Hence, an estimate of the condition number $\kappa$ can be central in the convergence analysis. Moreover, independence of the condition number w.r.t.\ problem parameters allows for robust convergence rates in the considered system class.
  The condition number of a symmetric linear operator $A: \calU \to\calU^*$ with respect to a preconditioner $Q:\calU \to \calU^*$ can be computed as follows. Assume that there are $0 <m \le M$, such that 
  \[
    m \langle v,Qv\rangle_{\calU,\calU^*} \le \langle v,Av\rangle_{\calU,\calU^*} \le M \langle v, Q v\rangle_{\calU,\calU^*} \quad \forall v\in \calU.
  \]
  If the constants in these estimates are sharp, then the condition number of $A$ relative to $Q$ is given as
  \[
    \kappa(A) := \frac{M}{m}.
  \]

\subsection{Gradient method for the reduced problem}\label{subsec:gradintro}
Due to continuous invertibility of the state operator associated with the constraint, we may eliminate the state via $x = \calA^{-1}\calB u + \calA^{-1}f$ such that we may equivalently solve the unconstrained problem
\begin{equation}\label{eq:redprob}
    \min_{u\in \calU} h(u) := \frac12\|\calC(\calA^{-1}\calB u + \calA^{-1}f) - y_\mathrm{ref}\|_\calY^2 + \frac\alpha2 \|u\|_\calU^2.
\end{equation}
Due to strict convexity of the quadratic function $h:\calU \to \R$, the solution of this minimization problem amounts to solving the linear system
\begin{equation}\label{eq:cg}
    0 = h'(u^*) = \left((\calC\calA^{-1}\calB)^*\Ry\calC\calA^{-1}\calB+ \alpha \Ru \right)u^* + (\calC\calA^{-1}\calB)^*\Ry(\calC\calA^{-1}f - y_\mathrm{ref}).
\end{equation}
Here, $\Ry: \calY \to \calY^*$ and $\Ru: \calU \to \calU^*$ denote the Riesz isomorphisms corresponding to the Hilbert input and output spaces, respectively. A gradient direction can be computed by solving the Riesz equation on $\calU$:
\[
 \nabla h(u) = \Ru^{-1}h'(u).
\]
As the system \eqref{eq:cg} is governed by a self-adjoint operator 
\begin{equation*}
    \calM_\mathrm{red} : \calU \to \calU^*, \qquad \calM_\mathrm{red}:= (\calC\calA^{-1}\calB)^*\Ry\calC\calA^{-1}\calB+ \alpha \Ru,
\end{equation*}
a suitable method to solve it is the conjugate gradient method, where in each iteration the evaluation of the governing operator $\calM_\mathrm{red}$ involves one solution of the state equation (due to $\calA^{-1}$) and one solution of the adjoint equation (due to $\calA^{-*}$). To assess the convergence behavior of a conjugate gradient method applied to \eqref{eq:cg}, we estimate the relevant condition number. %

For the self-adjoint operator defined above, we obtain the following well-known bound on the condition number, see e.g.~\cite{lubkoll2017affine}. Note that the proof follows also as a particular case of Proposition~\ref{prop:grad_cond_feedbacked}.
\begin{proposition}\label{prop:gradient:condition}
        With the norm of the control-to-observation operator $\sigma_0=\|\calC \calA^{-1}\calB\|_{\calU\to \calY}$ as defined in \eqref{eq:norm-c-to-obs}, we have for all $u\in \calU$
        \begin{equation*}
            \alpha\|u\|^2_\calU \leq \langle u,\calM_\mathrm{red} u\rangle_{\calU,\calU^*} \leq (\alpha + \sigma_0^2)\|u\|^2_\calU,
        \end{equation*}
        which implies the following condition number estimate:
\begin{equation}\label{eq:kappa_estimate_original}
            \kappa(\calM_\mathrm{red}) \leq 1 + \frac{\sigma_0^2}{\alpha}.
        \end{equation}
\end{proposition}
This result %
shows that the bound on the condition number strongly depends on the inverse operator $\|\calA^{-1}\|_{\calP^* \to \calX}$. In Section~\ref{sec:appl}  we will show for various applications of elliptic, parabolic and hyperbolic equations that this operator norm may become prohibitively large and that the feedback transformation, that is, in essence replacing $\calA$ by $\calA-\calB\calK$ as in \eqref{eq:OCP2} may alleviate this effect.

\subsection{Preconditioned projected conjugate gradients}\label{subsec:ppcgintro}
The second approach used to solve the optimal control problem \eqref{eq:OCP} builds upon the first-order optimality condition of the constrained problem without eliminating the state. 

To this end, we recall the optimality conditions of the abstract optimal control problem \eqref{eq:OCP}, see, e.g., \cite[Theorem 1.1]{schiela2013concise}. Therein, the central assumption is that the constraint operator $[\calA,-\calB]$ has closed range, which in our setting follows from the fact that $\calA$ is an isomorphism, hence surjective such that $[\calA,-\calB]$ is surjective (and thus has closed range) as well.
Let $(x,u)\in \calX \times \calU$ be optimal for \eqref{eq:OCP}. Then, there is an adjoint state $p\in \calP^*$ such that 
\begin{equation}\label{eq:optsys}
\calM_\mathrm{opt}\begin{bmatrix}
    x\\u\\p
\end{bmatrix}=
\begin{bmatrix}
    \calC^*\Ry\calC & 0 & \calA^*\\
    0 & \alpha \Ru & -\calB^*\\
    \calA& -\calB &0
\end{bmatrix}\begin{bmatrix}
    x\\u\\p
\end{bmatrix}
= 
\begin{bmatrix}
    \calC^*y_\mathrm{ref}\\
    0\\
    f
\end{bmatrix}.
\end{equation}
While this system is symmetric, it is indefinite due to the presence of the zero block in the bottom right. Hence, to apply a conjugate gradient method, one may incorporate a constraint preconditioner
\begin{equation}\label{eq:const_prec}
  \calQ = \begin{bmatrix}
    0 & 0& \calA^*\\
    0& \alpha\Ru & -\calB^*\\
    \calA & -\calB & 0
\end{bmatrix}
\end{equation}
which is invertible due to invertibility of $\calA$ (cf. \cite{keller2000constraint,schiela2014operator,rees2010optimal}).
Due to its triangular structure, application of its inverse decouples into one adjoint solve, the solution of the Riesz equation $\alpha \Ru w = v$ and one state solve. 

When endowed with this preconditioner, it is guaranteed that the state and control iterates of the conjugate gradient method evolve in the subspace $\ker [\calA,-\calB]\times \calP$ on which the governing operator $\calM_\mathrm{opt} : \calX\times\calU\times\calP \to (\calX\times\calU\times\calP)^*$ as defined in \eqref{eq:optsys} is positive definite.

When using a constraint preconditioner, the speed of convergence of PPCG is governed by the condition number of the top left $2\times2$ block of $\calM_\mathrm{opt}$ with respect to the top left $2\times 2$ block of $\calQ$ on $\ker [\calA,-\calB] = \{(\calA^{-1}\calB u, u)\,|\,u\in \calU\}$. 

    One may use this insight to actually show the equivalence of the two condition number estimates of Proposition~\ref{prop:gradient:condition} and Proposition~\ref{prop:ppcg:condition}. In particular, we note that for $z:= (x,u,p)\in \calZ := \calX \times \calU \times \calP$ with $(x,u)\in \ker [\calA,-\calB]$,
    \begin{align}\label{eq:itsthesame}
        \langle z,\calM_\mathrm{opt}z\rangle_{\calZ,\calZ^*} = \|\calC x\|^2_\calY + \alpha\|u\|^2_\calU = \|\calC \calA^{-1}\calB u\|^2_\calY + \alpha\|u\|^2_\calU = \langle u, \calM_\mathrm{red}u\rangle_\calU.
    \end{align}
    In this sense, the scalar product induced by the restriction of $\calM_\mathrm{opt}$ to $\ker [\calA,-\calB] \times \calP$ (as done for the PPCG method in Proposition~\ref{prop:ppcg:condition}) coincides with the scalar product induced by $\calM_\mathrm{red}$. Hence, we immediately obtain the following result.
\begin{proposition}\label{prop:ppcg:condition}
For all $z = (x,u,p)\in \calZ:=\calX \times \calU \times \calP$ with $(x,u)\in \ker [\calA,-\calB]$,
\begin{equation*}
    \langle z, \calQ z\rangle_{\calZ,\calZ^*} \leq \langle z,\calM_\mathrm{opt}z\rangle_{\calZ,\calZ^*} \leq \left(1+\tfrac{\sigma_0^2}{\alpha}\right)\langle z,\calQ z\rangle_{\calZ,\calZ^*} 
\end{equation*}
and hence we have the condition number estimate~\eqref{eq:condnumber} for $\kappa(\mathcal{M}_\mathrm{opt})$.
\end{proposition}

\subsection{Illustrative scalar example}\label{subsec:easyex}
In this part, we briefly provide a very simple scalar problem to illustrate the main idea of the suggested approach. %
For more complex problems, we refer to Section~\ref{sec:appl} where we illustrate for a wide range of partial differential equations (such as elliptic, hyperbolic and parabolic problems) how to suitably choose the feedback operator to obtain a small solution operator norm in the transformed problem. %
For now, let us consider the discrete-time problem
\begin{align}\label{eq:discopt}
\begin{split}
            \frac12\sum_{k=1}^{N} |x_k - 5|^2 + \frac\alpha2 \sum_{k=0}^{N-1}  |u_k|^2.
 \quad \mathrm{s.t.}\quad x_{k+1} &= ax_k + u_k, \ k=0,\ldots N-1, \\[-.3cm]
 x_0 &= 0.
\end{split}
\end{align}
Here, $\calU = \calX = \calP=\calY = \R^N$ and $0\leq a\in \R$ is a parameter to be specified.
It is easily seen that the control-to-state map $\mathcal{A}^{-1}$ corresponds to solving the forward iteration of the discrete time constraint, that is,
\begin{align}\label{eq:gradient:calA}
    \begin{smallbmatrix}
        x_{1}\\x_{2}\\\vdots\\x_{N}
    \end{smallbmatrix}
     = 
     \begin{smallbmatrix}
       u_0\\
       a x_0 + u_1\\
       \vdots \\
       a x_{N-1} + u_{N-1}
     \end{smallbmatrix}
     =
     \begin{smallbmatrix}
        1 & 0 & 0 & \ldots & 0 \\
        a & 1 & 0 & \ldots & 0 \\
        a^2 & a & 1 & \ldots & 0 \\
        \vdots & \vdots & \ddots & \vdots & \vdots  \\
         a^{N-1}  & \ldots & \ldots & a & 1
    \end{smallbmatrix}
    \begin{smallbmatrix}
        u_0\\ u_1\\\vdots \\ u_{N-1}
    \end{smallbmatrix} =: T_N u
\end{align}
and input resp.\ output operators are just identity matrices $\calC = \calB = I \in \R^{N\times N}$.
Hence, one may easily show that (see Appendix~\ref{subsec:toeplitz})
{\begin{small}
\begin{equation}\label{eq:toeplitz}
\sigma_0=\|\calC \calA^{-1} \calB\|_{\R^N\to \R^N} = \|\calA^{-1}\|_{\R^N\to \R^N} 
= 
\begin{cases}
    \mathcal{O}(1) & a < 1, \\
    \mathcal{O}(N) & a = 1, \\   
    \mathcal{O}(a^N) & a > 1.
\end{cases}
\end{equation}
\end{small}}
Hence, if the original problem in the constraint is unstable ($a>1$), the operator norm deteriorates for larger horizons. This reflects the instability of the uncontrolled problem when $a>1$. In this case, $\mathcal{A}^{-1}$ and its adjoint significantly amplify any kind of errors, which relates to the condition number of the governing matrix derived in Proposition~\ref{prop:gradient:condition} or Proposition~\ref{prop:ppcg:condition}. 

However, we may easily stabilize the constraint by introducing a feedback in the sense of \eqref{eq:feedback} via $u=kx+v$ with $k=-a+0.5$ such that $\widetilde \calA$ is governed by the stable equation
\begin{equation*}
    x_{k+1} = ax_k + u = ax_k + (-a+0.5)x_k +v = 0.5x_k + v.
\end{equation*}
Consequently, the associated control-to-state map denoted by $\widetilde \calA^{-1}$ satisfies \eqref{eq:gradient:calA} with $a= 0.5$ such that $\|{\widetilde{\calA}}^{-1}\| \sim \max\{1,|0.5|^{N-1}\} = 1$. In particular, the solution operator norm is bounded uniformly in the time horizon $N$.

We briefly illustrate the effect of the feedback for the gradient method introduced in Subsection~\ref{subsec:gradintro} in Figure~\ref{fig:grad:disc}. We depict the relative error of the control over the iterations for different parameters $(a,\alpha)$ of the problem \eqref{eq:discopt} with horizon $N=100$. We observe that the transformation leads to a better convergence even for the stable problem in the top row of Figure~\ref{fig:grad:disc}. This effect gets more drastic, when going to an unstable problem (bottom row) or when decreasing the Tikhonov parameter~$\alpha$. Importantly, as expected due to the feedback, the convergence behavior for the transformed problem is very robust w.r.t.\ the stability parameter $a$ when comparing the red lines of the top and bottom row of Figure~\ref{fig:grad:disc}. This clearly is not the case for the original problem. Last, we note that, when increasing the instability parameter to $a=1.5$ or increasing the time horizon, the gradient method for the original problem terminates with an overflow error, while the one for the feedback-transformed problem still converges.

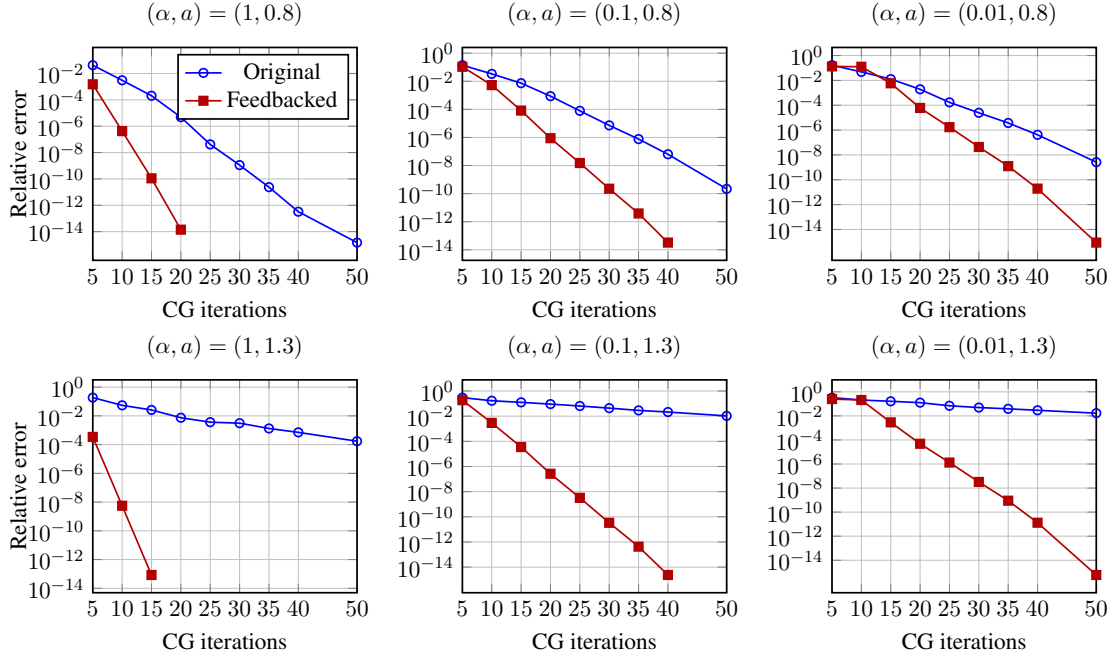
\begin{figure}[htb]
    \centering
            \scalebox{0.8}{%
\begin{tikzpicture}

\begin{axis}[
    width=0.35\textwidth,
    height=0.3\textwidth,
    xlabel={CG iterations},
    ylabel={Relative error},
    xmin=5, xmax=50,
    xtick={5,10,15,20,25,30,35,40,50},
    ymode=log,
    ytickten={0,-2,-4,-6,-8,-10,-12,-14},
    grid=both,
    legend pos=north east,
    thick,
    title={$(\alpha, a) = (1,0.8)$}
]
\addplot[mark=o, blue] coordinates {
    (5, 0.04235891123124236)
    (10, 0.003048305976226194)
    (15, 0.00020441133789705587)
    (20, 4.822765685728518e-06)
    (25, 4.13927970851602e-08)
    (30, 1.1200188109752822e-09)
    (35, 2.3738785083640106e-11)
    (40, 3.26343209155649e-13)
    (50, 1.4792847974608376e-15)
};
\addlegendentry{Original}
\addplot[mark=square*, red!70!black] coordinates {
    (5, 0.0015160728114093625)
    (10, 4.236997471622918e-07)
    (15, 1.1119691408312645e-10)
    (20, 1.4134736097013663e-14)
    (25, 0.0)
    (30, 0.0)
    (35, 0.0)
    (40, 0.0)
    (50, 0.0)
};
\addlegendentry{Feedbacked}
\end{axis}

\end{tikzpicture}
}
        \scalebox{0.8}{%
\begin{tikzpicture}

\begin{axis}[
    width=0.35\textwidth,
    height=0.3\textwidth,
    xlabel={CG iterations},
    xmin=5, xmax=50,
    xtick={5,10,15,20,25,30,35,40,50},
    ymode=log,
    ytickten={0,-2,-4,-6,-8,-10,-12,-14},
    grid=both,
    legend pos=north east,
    thick,
    title={$(\alpha, a) = (0.1,0.8)$}
]
\addplot[mark=o, blue] coordinates {
    (5, 0.13873140850197863)
    (10, 0.033373696711890534)
    (15, 0.007316802717585855)
    (20, 0.0008663185837618226)
    (25, 7.82933604391047e-05)
    (30, 7.220154437926613e-06)
    (35, 7.649034185524239e-07)
    (40, 6.344613152002568e-08)
    (50, 2.163884276765665e-10)
};
\addplot[mark=square*, red!70!black] coordinates {
    (5, 0.10716200879746095)
    (10, 0.005312454137016235)
    (15, 8.212561474721467e-05)
    (20, 8.843479687164214e-07)
    (25, 1.5039800822896697e-08)
    (30, 2.228575657499062e-10)
    (35, 3.881362987384525e-12)
    (40, 3.26598958221408e-14)
    (50, 0.0)
};
\end{axis}

\end{tikzpicture}
}   
    \scalebox{0.8}{%
\begin{tikzpicture}

\begin{axis}[
    width=0.35\textwidth,
    height=0.3\textwidth,
    xlabel={CG iterations},
    xmin=5, xmax=50,
    xtick={5,10,15,20,25,30,35,40,50},
    ymode=log,
    ytickten={0,-2,-4,-6,-8,-10,-12,-14},
    grid=both,
    legend pos=north east,
    thick,
    title={$(\alpha, a) = (0.01,0.8)$}
]
\addplot[mark=o, blue] coordinates {
    (5, 0.16678349683608373)
    (10, 0.04825670189579396)
    (15, 0.012688720418909182)
    (20, 0.001921679162908335)
    (25, 0.00017036795449768468)
    (30, 2.4590798156442584e-05)
    (35, 3.745573523519482e-06)
    (40, 4.153519524996954e-07)
    (50, 2.6170094498143915e-09)
};
\addplot[mark=square*, red!70!black] coordinates {
    (5, 0.13336792083522656)
    (10, 0.12489823544077416)
    (15, 0.005857259927235613)
    (20, 6.046053784754301e-05)
    (25, 1.7154489082690188e-06)
    (30, 4.360990372545683e-08)
    (35, 1.2515376050397297e-09)
    (40, 1.9342862406903293e-11)
    (50, 8.761101469586852e-16)
};
\end{axis}

\end{tikzpicture}
}\\
        \scalebox{0.8}{%
\begin{tikzpicture}

\begin{axis}[
    width=0.35\textwidth,
    height=0.3\textwidth,
    xlabel={CG iterations},
    ylabel={Relative error},
    xmin=5, xmax=50,
    xtick={5,10,15,20,25,30,35,40,50},
    ymode=log,
    ytickten={0,-2,-4,-6,-8,-10,-12,-14},
    grid=both,
    legend pos=north east,
    thick,
    title={$(\alpha, a) = (1,1.3)$}
]
\addplot[mark=o, blue] coordinates {
    (5, 0.1863237056771352)
    (10, 0.05402498057324269)
    (15, 0.025819768885334955)
    (20, 0.007396670083771834)
    (25, 0.0036331373561403614)
    (30, 0.003097873548688374)
    (35, 0.0013303865047986766)
    (40, 0.0007100649639852083)
    (50, 0.000172465082642948)
};
\addplot[mark=square*, red!70!black] coordinates {
    (5, 0.0003424978231342578)
    (10, 5.569385315404921e-09)
    (15, 8.448807074076804e-14)
    (20, 0.0)
    (25, 0.0)
    (30, 0.0)
    (35, 0.0)
    (40, 0.0)
    (50, 0.0)
};
\end{axis}

\end{tikzpicture}
}
    \scalebox{0.8}{%
\begin{tikzpicture}

\begin{axis}[
    width=0.35\textwidth,
    height=0.3\textwidth,
    xlabel={CG iterations},
    xmin=5, xmax=50,
    xtick={5,10,15,20,25,30,35,40,50},
    ymode=log,
    ytickten={0,-2,-4,-6,-8,-10,-12,-14},
    grid=both,
    legend pos=north east,
    thick,
    title={$(\alpha, a) = (0.1,1.3)$}
]
\addplot[mark=o, blue] coordinates {
    (5, 0.2983517096230032)
    (10, 0.1682979081707932)
    (15, 0.1251265478918416)
    (20, 0.09101172027460903)
    (25, 0.06419015842799952)
    (30, 0.04354716944113693)
    (35, 0.028678824238280026)
    (40, 0.02125545788408028)
    (50, 0.010527063137735462)
};
\addplot[mark=square*, red!70!black] coordinates {
    (5, 0.1867628466028816)
    (10, 0.0029184627188420047)
    (15, 3.510130067638876e-05)
    (20, 2.641190594749425e-07)
    (25, 3.2194583047995846e-09)
    (30, 3.3418561757682515e-11)
    (35, 4.22837620019158e-13)
    (40, 2.364914179553959e-15)
    (50, 0.0)
};
\end{axis}

\end{tikzpicture}
}
        \scalebox{0.8}{%
\begin{tikzpicture}

\begin{axis}[
    width=0.35\textwidth,
    height=0.3\textwidth,
    xlabel={CG iterations},
    xmin=5, xmax=50,
    xtick={5,10,15,20,25,30,35,40,50},
    ymode=log,
    ytickten={0,-2,-4,-6,-8,-10,-12,-14},
    grid=both,
    legend pos=north east,
    thick,
    title={$(\alpha, a) = (0.01,1.3)$}
]
\addplot[mark=o, blue] coordinates {
    (5, 0.33089969310067624)
    (10, 0.20740647111838847)
    (15, 0.16329168038271794)
    (20, 0.12629688256174734)
    (25, 0.06946081136703502)
    (30, 0.04930949759744233)
    (35, 0.03862125331435998)
    (40, 0.029286326056520557)
    (50, 0.016787324754460722)
};
\addplot[mark=square*, red!70!black] coordinates {
    (5, 0.24990581984559593)
    (10, 0.20920063695693922)
    (15, 0.0028999024051245222)
    (20, 4.831438274975951e-05)
    (25, 1.3211576936712032e-06)
    (30, 3.226726462338523e-08)
    (35, 8.934366160900636e-10)
    (40, 1.318774723046221e-11)
    (50, 5.941522260864888e-16)
};
\end{axis}

\end{tikzpicture}
}   
    \caption{Time-discrete problem: Relative error over CG iterations for the condensed system \eqref{eq:cg}. Top: Stable problem with $a=0.8$. Bottom: Unstable problem with $a=1.3$.}
    \label{fig:grad:disc}
\end{figure}

\section{Solution of the feedback-transformed problem}\label{sec:feedbacked_problem}
In the previous section, we have introduced two numerical algorithms and illustrated how they might benefit from feedback stabilizations. However, the algorithms suggested in Subsections~\ref{subsec:gradintro} and \ref{subsec:ppcgintro} particularly leveraged the fact that there are no mixed terms involving both $x$ and $u$ in the cost functional. For the gradient method of the reduced problems, this implies that the governing operator $\calM_\mathrm{red}$ of \eqref{eq:cg} consists of only two summands of which only one involves state and adjoint solves. For the PPCG method, as a consequence of the absence of coupling terms in the cost functional, the top left $2\times 2$ block of $\calM_{\mathrm{opt}}$ was block-diagonal.

In this part, we thus illustrate how to apply these methods to the feedback-transformed problem. In particular, we provide bounds on the condition number of the transformed problem. Before analyzing the algorithms, however, we briefly analyze the feedbacked optimal control problem. To this end we note that in view of optimality conditions (see \eqref{eq:optsys}), it is usually very helpful if the constraint operator has closed range. Moreover, we provide a relation of the optimal solutions of the original problem~\eqref{eq:OCP} and its feedback-transformed counterpart~\eqref{eq:OCP2}.
\begin{proposition}\label{prop:theoretischallesgut}
Let $\calK\in L(\calX,\calU)$. Then the following hold.
\begin{enumerate}[label=(\roman*)]
    \item  If $\calA$ is surjective, then the feedbacked constraint operator %
    \begin{equation*}
    [\widetilde\calA,-\calB] \coloneqq [\calA - \calB \calK, -\calB] : \calX\times \calU \to \calP^*
    \end{equation*}
    is surjective.
    \item  The state-control pair $(x,u)$ is optimal for \eqref{eq:OCP} if and only if the transformed state-control pair $(x,u-\calK x)$ is optimal for \eqref{eq:OCP2}. Equivalently, $(x,v)$ is optimal for \eqref{eq:OCP2} if and only if $(x,\calK x + v)$ is optimal for \eqref{eq:OCP}.
\end{enumerate}
\end{proposition}
\begin{proof}
  (i): Let $f\in \calP^*$. As $\calA$ is surjective, there is $x\in \calX $ such that $\calA x = f$. Setting $v= -\calK x$, $[\widetilde \calA, \calB] \left[\begin{smallmatrix}
        x\\v
    \end{smallmatrix}\right] = (\calA - \calB \calK)x + \calB\calK x = \calA x = f$ and the claim follows.

(ii):  This follows from the fact that the variable substitution is just a coordinate transformation given by the invertible linear map $\left[\begin{smallmatrix}
        x\\v
    \end{smallmatrix}\right] = \calW \left[\begin{smallmatrix}
        x\\u
    \end{smallmatrix}\right]$ with $\calW=\left[\begin{smallmatrix}
        I & 0 \\
        \calK & I
    \end{smallmatrix}\right]$.
\end{proof}

We now state the main assumption on the choice of the feedback.
\begin{assumption}\label{ass:feedback}
The feedback $\calK \in L(\calX,\calU)$ is chosen such that:
\begin{enumerate}[label=(\roman*)]
    \item The feedbacked state operator $\tcalA = \calA - \calB \calK: \calX \to \calP^*$ is %
    invertible, and the operator norm $\sigma_\calK=\|\calC\widetilde{\calA}^{-1}\calB\|_{\calU \to \calY}$ is finite\label{ass:iso}
    \item There is $\delta_\calK > 0$ such that $\|\calK x\|_\calU \leq \delta_\calK \|\calC x\|_\calY$ for all $x\in \calX$. \label{ass:smallness}
\end{enumerate}
\end{assumption}
We briefly comment on this assumption. First, we note that Assumption~\ref{ass:feedback} can always be satisfied by choosing $\calK=0$. In this case, however, the transformed problem is identical to the original problem. When choosing a feedback operator $0\neq \calK\in L(\calX,\calU)$, Assumption~\ref{ass:feedback}\ref{ass:iso} is to ensure that the methods are applicable as both build upon solving the constraint for the state. Further, Assumption~\ref{ass:feedback}\ref{ass:smallness} ensures smallness of the feedback w.r.t.\ the cost. If the observation operator $\calC: \calX \to \calY$ has closed range, this may be characterized geometrically in terms of a kernel inclusion. Precisely, it is easily checked that if $\ran \calC$ is closed, then Assumption~\ref{ass:feedback}\ref{ass:smallness} is equivalent to $\ker \calC \subset \ker \calK.$

In the following Subsections~\ref{subsec:gradfeedbacked} and \ref{subsec:ppcgfeedbacked} we show how to use the discussed numerical methods for the feedback-transformed problems and provide bounds on the condition number.

\subsection{Gradient method for the reduced problem}\label{subsec:gradfeedbacked}
In this subsection, we analyze the gradient method for the condensed problem when applied to the feedback-transformed optimal control problem~\eqref{eq:OCP2}.

To this end, we denote $\widetilde{\calA} = \mathcal{A}-\mathcal{B}\mathcal{K}$ such that eliminating the state by the control in \eqref{eq:OCP2} via $x = \widetilde{\mathcal{A}}^{-1} (\calB v + f)$ we obtain the reduced transformed problem
\begin{equation*}
    \min_{v\in \calU} \tilde h(v) := \frac12\|\calC({\tcalA}^{-1}\calB v + {\tcalA}^{-1}f) - y_\mathrm{ref}\|_\calY^2 + \frac\alpha2 \|\mathcal{K}({\tcalA}^{-1}\calB v + {\tcalA}^{-1}f)+v\|_\calU^2.
\end{equation*}
Expanding the squares, this cost functional is (up to constant terms) given by
\begin{align*}
    \tilde h(v) &= \frac12 \|\calC{\tcalA}^{-1}\calB v\|_\calY^2 + \langle \calC{\tcalA}^{-1}\calB v,\calC{\tcalA}^{-1}f-y_\mathrm{ref}  \rangle_{\calY} \\
    & \qquad  + \frac{\alpha}{2}\| \mathcal{K}{\tcalA}^{-1}\calB v + v\|_\calU^2  + \alpha\langle \mathcal{K}{\tcalA}^{-1}\calB v + v,\mathcal{K}\tcalA^{-1}f\rangle
\end{align*}
leading to the optimality condition 
\begin{align}\label{eq:cg2}
\begin{split}
       0 = \nabla \tilde h(v^*) &= \left((\calC\tcalA^{-1}\calB)^*\Ry\calC\tcalA^{-1}\calB+ \alpha (\mathcal{K}\tcalA^{-1}\calB + I)^*\Ru (\mathcal{K}\tcalA^{-1}\calB + I)\right)v^* \\
    & \qquad  + (\calC\calA^{-1}\calB)^*\Ry(\calC\calA^{-1}f - y_\mathrm{ref}) + \alpha(\mathcal{K}\tcalA^{-1}\calB + I)^*\Ru \mathcal{K}\tcalA^{-1}f.
\end{split}
\end{align}
Hence, the transformed governing operator is $ \widetilde{\mathcal{M}}_\mathrm{red} : \calU \to \calU^*$ defined by
\begin{equation}
    \widetilde{\mathcal{M}}_\mathrm{red} := (\calC\tcalA^{-1}\calB)^*\Ry\calC\tcalA^{-1}\calB+ \alpha (\mathcal{K}\tcalA^{-1}\calB + I)^*\Ru (\mathcal{K}\tcalA^{-1}\calB + I).
\end{equation}

Before discussing possible improvements of the condition number, we briefly discuss the complexity of applying $\widetilde{\calM}_\mathrm{red}$ in comparison to the original operator $\calM_\mathrm{red}$.
 We may organize the application of $\widetilde{\mathcal{M}}_\mathrm{red}$ to $v\in \calU$ as follows:
First we compute the state and the original control via $y = \tcalA^{-1}\calB v$ and $u= \Ru(\calK y+v)$. 
Then we may express the evaluation as $\widetilde{\mathcal{M}}_\mathrm{red}v = \calB^*\tcalA^{-*}(\calC^*\Ry \calC y+\alpha \mathcal{K}^* u) +  \alpha u$.
In this way, just as in the original problem governed by $\calM_{\mathrm{red}}$, one state and one adjoint solve are sufficient to compute a gradient step. The only additional cost, compared to the case $\calK=0$, is the application of the feedback operator and its adjoint. 

We now prove a condition number estimate for the feedbacked reduced optimality system.
\begin{proposition}\label{prop:grad_cond_feedbacked}
Let Assumption~\ref{ass:feedback} hold. Then for all $v\in \calU$      
        \begin{equation*}
            \frac{\alpha}{1+\alpha\delta_\calK^{2}} \|v\|^2_\calU \leq \langle v,\widetilde{\calM}_\mathrm{red} v\rangle_{\calU,\calU^*} 
            \leq \left(
\sigma_\calK^2
+
\alpha\left(1+\delta_\calK \sigma_\calK\right)^2
\right) \|v\|^2_\calU.
        \end{equation*}
        For the condition number we obtain:
\begin{equation}\label{eq:condnumber}
        \kappa(\widetilde{\calM}_\mathrm{red}) \le (1+\alpha \delta_\calK^2)\left(
\frac{\sigma_\calK^2}{\alpha}
+
\bigl(1+\delta_\calK \sigma_\calK\bigr)^2
\right).
\end{equation}
\end{proposition}
\begin{proof}
The upper bound follows directly from 
\begin{align*}
\langle v,\widetilde{\calM}_\mathrm{red}v\rangle_\calU
&=
\|\calC\tcalA^{-1}\calB v\|_\calY^2
+
\alpha\|\calK\tcalA^{-1}\calB v+v\|_\calU^2 \\
&\le
\sigma_\calK^2\|v\|_\calU^2
+
\alpha\bigl(\|\calK\tcalA^{-1}\calB v\|_\calU+\|v\|_\calU\bigr)^2 \\
&\le
\sigma_\calK^2\|v\|_\calU^2
+
\alpha\bigl(1+\delta_\calK \sigma_\calK\bigr)^2\|v\|_\calU^2 =
\Bigl(
\sigma_\calK^2
+
\alpha\bigl(1+\delta_\calK \sigma_\calK\bigr)^2
\Bigr)
\|v\|_\calU^2 .
\end{align*}
For the lower bound, we expand the square such that
\begin{align}\label{eq:expandedsquare}
        \langle v,\widetilde{\calM}_\mathrm{red}v \rangle_\calU = \|\calC \tcalA^{-1}\calB v\|_{\calY}^2 + \alpha \left(\|\calK\tcalA^{-1}\calB v \|_\calU^2 + 2\langle v, \calK\tcalA^{-1}\calB v\rangle_\calU + \|v\|_\calU^2\right).
    \end{align}
By Assumption~\ref{ass:feedback}\ref{ass:smallness}, we have $\delta_\calK^{-2}\|\calK x\|_\calU^2 \leq \|\calC x\|^2_\calY$ for all $x\in \calX$. Further note that the Cauchy-Schwarz inequality and the binomial formula yield 
\begin{align*}
    2 \langle v,\calK\tcalA^{-1}\calB v \rangle_\calU \leq 2\|v\|_\calU \|\calK \tcalA^{-1}\calB v\|_\calU \leq c\|v\|_\calU^2 + \tfrac{1}{c} \|\calK \tcalA^{-1}\calB v\|_\calU^2
\end{align*}
for any $c>0$. Inserting these inequalities with $c=\frac{\alpha}{\delta_\calK^{-2}+\alpha}$ into \eqref{eq:expandedsquare}, we get
\begin{align*}
   &  \langle v,\widetilde{\calM}_\mathrm{red}v \rangle_\calU \\
   &\geq (\delta_\calK^{-2}+\alpha)\|\calK\tcalA^{-1}\calB v \|_\calU^2 + \alpha \left(- \frac{\alpha}{\delta_\calK^{-2}+\alpha}\|v\|^2_\calU - \frac{\delta_\calK^{-2} + \alpha}{\alpha} \|\calK\tcalA^{-1}\calB v\|^2_\calU   + \|v\|_\calU^2\right)\\ 
    &= \alpha \left( 1-\frac{\alpha}{\delta_\calK^{-2}+\alpha}\right) \|v\|^2_\calU =\frac{\alpha}{1+\alpha\delta_\calK^{2}} \|v\|^2_\calU.
\end{align*}
\end{proof}

We briefly mention that while one may leverage the feedback to reduce a possible instability manifesting in the solution operator norm, the feedback bound $\delta_\calK$ also appears explicitly in the estimates (e.g.~\eqref{eq:condnumber}) such that too large feedbacks might worsen the condition number. Hence, it is in practice crucial to balance the feedback size and the solution operator norm. If the solution operator norm $\sigma_\calK$ may be linked with the relative bound on the feedback $\delta_\calK$ of Assumption~\ref{ass:feedback}\ref{ass:smallness}, then one may optimize the upper bound \eqref{eq:condnumber} w.r.t.\ the feedback, see Figure~\ref{fig:spoileropt}. 
\begin{figure}[htb]
    \centering
    \scalebox{.8}{%
\begin{tikzpicture}
\begin{axis}[
    width=0.46\textwidth,
    height=0.25\textwidth,
    xmode=log,
    ymode=log,
    xmin=0.0008333333333333334, xmax=1200.0,
    xlabel={$\delta$},
    ymin = 10,
    ymax = 10^5,
    ylabel={Error after 5 iter.},
    axis y line*=left,
    axis x line*=bottom,
    grid=both,
    thick,
]
\addplot[mark=o,] coordinates {
    (0.001, 11005.505832069923)
    (0.01, 2005.0184081362897)
    (0.1, 223.51133072721447)
    (0.5, 44.84996304432988)
    (10, 25.435057307003365)
    (70.71067811865476, 24.69170173360892)
    (100, 63.63691170080697)
    (1000.0, 1459.8254317089202)
};
\addplot[red, dashed] coordinates {
    (70.71067811865476, 10)
    (70.71067811865476, 10^5)
};
\end{axis}
\end{tikzpicture}
}\hspace{.8cm}
    \scalebox{.8}{%
\begin{tikzpicture}
\begin{axis}[
    width=0.46\textwidth,
    height=0.25\textwidth,
    xmode=log,
    ymode=log,
    xmin=0.0008333333333333334, xmax=1200.0,
    xlabel={$\delta$},
    ylabel={Error after 20 iter.},
    ymin = 1e-9,
    ymax = 10,
    axis y line*=left,
    axis x line*=bottom,
    grid=both,
    thick,
]
\addplot[mark=o] coordinates {
    (0.001, 0.017657462427447732)
    (0.01, 0.017317452137050986)
    (0.1, 0.009185875316701463)
    (1, 4.4923390355135003e-07)
    (1.43342685048752, 4.812461194998673e-09)
    (10, 3.801525271934263e-08)
    (100, 0.09475623936389407)
    (1000, 0.6742085932727813)
};
\addplot[red, dashed] coordinates {
    (1.43342685048752, 1e-09)
    (1.43342685048752, 10)
};
\end{axis}

\end{tikzpicture}
}
    \caption{Error after fixed number of conjugate gradient iterations for varying feedback strengths parameterized by $\delta_\calK$. The vertical red line depicts the theoretically optimized feedback strength $\delta_\calK^*$ minimizing the bound on the condition number. Left: Elliptic problem, right: parabolic problem.}
    \label{fig:spoileropt}
\end{figure}
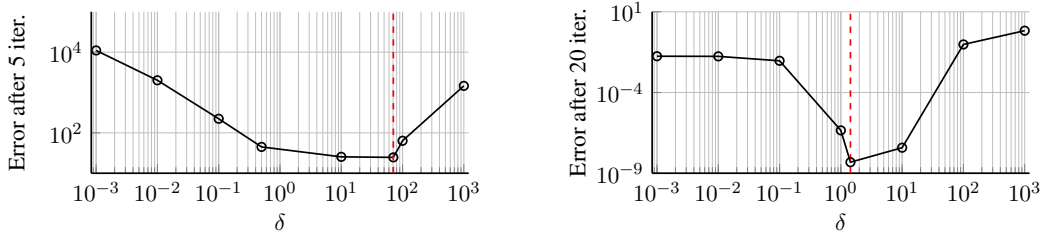
We will discuss these examples and how to obtain the optimal feedback strength $\delta_\calK^*$ that minimizes condition number bound \eqref{eq:condnumber} in detail in Section~\ref{sec:appl}. However, we already see here that the theoretical optimization of the condition number leads to the best performance of the numerical method, that is, the smallest error after a fixed number of iterations.

Next, we analyze the PPCG method for the feedbacked problem.
\subsection{Preconditioned projected conjugate gradients} \label{subsec:ppcgfeedbacked}
To apply the PPCG method as introduced in Subsection~\ref{subsec:ppcgintro} to the transformed problem \eqref{eq:OCP2}, we first deduce the optimality system for the transformed problem. To this end, observe by Proposition~\ref{prop:theoretischallesgut} that the transformed constraint operator has closed range such that we may again apply the abstract optimality conditions from \cite[Theorem 1.1]{schiela2013concise}.

Let $(x,v) \in \calX \times \calU$ be optimal for \eqref{eq:OCP2}. Then, there is $p\in \calP^*$ such that 
\begin{equation}\label{eq:optsys2}
\begin{bmatrix}
    \calC^*\Ry\calC + \alpha \calK^*\Ru \calK & \alpha \calK^*\Ru & (\calA - \calB \calK)^*\\
    \alpha \Ru \calK &  \alpha \Ru & -\calB^*\\
    (\calA - \calB \calK)& -\calB &0
\end{bmatrix}\begin{bmatrix}
    x\\v\\p
\end{bmatrix}
= 
\begin{bmatrix}
    \calC^*y_\mathrm{ref}\\
    0\\
    f
\end{bmatrix}.
\end{equation}
We briefly comment on the relation between the optimality systems of \eqref{eq:OCP} and \eqref{eq:OCP2}.
\begin{remark}
The \emph{feedbacked optimality system} \eqref{eq:optsys2} results from a congruence transformation of \eqref{eq:optsys}. More precisely,
$
    \calW^* \calM_\mathrm{opt} \calW = \widetilde\calM_{\mathrm{opt}}$ with $\calW := \begin{smallbmatrix}
      I & & \\
      \calK & I &\\
    ~  & & I
    \end{smallbmatrix}$.

\end{remark}
In the transformed case, a suitable constraint preconditioner for \eqref{eq:optsys2} is given by
\[
  \widetilde \calQ = \begin{bmatrix}
    0 & 0 & (\calA-\calB \calK)^*\\
    0 & \alpha\Ru & -\calB^*\\
    \calA-\calB\calK & -\calB & 0
\end{bmatrix}.
\]
\begin{proposition}\label{prop:grad_ppcg_feedbacked}
Let Assumption~\ref{ass:feedback} hold. Then for all $z = (x,v,p)\in \calZ:=\calX \times \calU \times \calP$ with $(x,v)\in \ker [\tcalA,-\calB]$ 
\begin{align*}
     \frac{1}{1+\alpha\delta_\calK^2} \langle z,  \widetilde\calQ z\rangle_{\calZ,\calZ^*} \leq \langle z,\widetilde{\calM}_\mathrm{opt}z\rangle_{\calZ,\calZ^*} \leq \left(
\frac{\sigma_\calK^2}{\alpha}
+
\left(1+\delta_\calK \sigma_\calK\right)^2
\right)\langle z, \widetilde \calQ z\rangle_{\calZ,\calZ^*}.
\end{align*}
and hence the condition number estimate \eqref{eq:condnumber} for $\kappa(\widetilde{\mathcal{M}}_\mathrm{opt})$.
\end{proposition}

\begin{proof}
The proof follows from Proposition~\ref{prop:grad_cond_feedbacked} and the counterpart of \eqref{eq:itsthesame} for the transformed case, i.e., $z:= (x,v,p)\in \calZ %
$ with $(x,v)\in \ker [\tcalA,-\calB]$,
    \begin{align*}
        \langle z,\widetilde{\calM}_\mathrm{opt}z\rangle_{\calZ,\calZ^*} &= \|\calC x\|^2_\calY + \alpha\|\calK x + v\|^2_\calU \\&= \|\calC \tcalA^{-1}\calB v\|^2_\calY + \alpha\|\calK \tcalA^{-1}\calB v+ v\|^2_\calU = \langle v, \widetilde{\calM}_\mathrm{red}v\rangle_\calU.
    \end{align*}
\end{proof}
We briefly comment on the estimates of Propositions~\ref{prop:grad_cond_feedbacked} and \ref{prop:grad_ppcg_feedbacked}. Therein, we observe a dependence of the condition number on the feedbacked solution operator $\tcalA^{-1} = (\calA - \calB\calK)^{-1}$ instead of the original solution operator $\calA^{-1}$. Hence, if the feedback is chosen suitably to reduce the solution operator norm,  the condition number may be drastically improved. This improvement may come from stabilization of a time-dependent problem as done for the illustrating scalar example of Subsection~\ref{subsec:easyex} and as will be done for general ODEs, heat and wave equations in the subsequent Section~\ref{sec:appl}. For stationary problems, such as elliptic partial differential equations, we will show how to enlarge the coercivity constant of the associated bilinear form.

\section{Applications}\label{sec:appl}
In this part, we provide several important applications of the above developed abstract approach. In Subsection~\ref{subsec:fd} we discuss the case of ODE constraints to illustrate the main mechanism of the feedback stabilization in view of linear system theory. Then, in Subsection~\ref{subsec:elliptic} we move to elliptic PDEs, where feedbacks are used to shape the coercivity of the governing main operator. Last, we move again to time-dependent problems and apply our method to parabolic equations in Subsection~\ref{subsec:parabolic} and to hyperbolic equations in Subsection~\ref{subsec:hyperbolic}.
\subsection{Finite-dimensional control systems}\label{subsec:fd}
In this part, we consider linear finite-dimensional control systems
\begin{equation}\label{eq:fin_dim_cs}
    \dot x(t) = A x(t) + Bu(t)  \quad t\in [0,T], \qquad x(0) = x_0,
\end{equation}
where $T>0$ is a time horizon, $A \in \R^{n \times n}$ denotes the \textit{system matrix}, $B \in \R^{m \times n}$ is a \textit{control input} matrix, $u: [0, T) \to \R^m$ is a \textit{control function} and $x: [0, T) \to \R^n$ is the associated state. For brevity, we refer to the control system \eqref{eq:fin_dim_cs} as the pair $(A,B)$. For every $x_0 \in \mathbb{R}^n$ and $u \in L^2(0,T;\mathbb{R}^m)$, \eqref{eq:fin_dim_cs} admits a unique \textit{mild solution}
$x \in H^1(0,T;\mathbb{R}^n) \subset C([0,T];\mathbb{R}^n)$
given by the \emph{variation-of-constants} formula
\begin{equation}\label{eq:var_of_const}
x(t) = e^{tA}x_0 + \int_0^t e^{(t-s)A} B u(s)\,\mathrm{d}s,
\quad t \in [0,T].
\end{equation}
Moreover, $x$ satisfies \eqref{eq:fin_dim_cs} almost everywhere. Note that, due to \cite[Thm. 2.1.7]{curtain2020introduction},  for any matrix $A\in \R^{n\times n}$ there exist $\omega \in \R$ and $M \geq 1$ such that
    \begin{equation}\label{eq:Gronwall}
        \left\|e^{t A} \right\|_{\R^n\to\R^n} \leq M e^{\omega t} , \qquad t \geq 0.
    \end{equation}
We briefly comment on the constants involved in this estimate in view of spectral properties of the matrix $A$.
\begin{remark}
  The constants in \eqref{eq:Gronwall} may be obtained by the \textit{growth bound} of $e^{tA}$ defined by $$\omega_0(A) \coloneqq \sup \{\re \lambda \, | \, \lambda \text{ is an eigenvalue of } A \}.$$ Then, \eqref{eq:Gronwall} holds for any $\omega > \omega_0(A)$, where, in general, $M$ depends on the chosen $\omega$. 

    If $A$ is normal, e.g.\ (skew-)symmetric, then \eqref{eq:Gronwall} holds with $M=1$ and $\omega = \omega_0(A)$; this may be seen using a diagonalization argument. In particular, for the skew-symmetric case,  $\|e^{tA}\|_{\R^n\to\R^n}\equiv 1$.
\end{remark}

To rewrite the control system \eqref{eq:fin_dim_cs} as an abstract constraint as in \eqref{eq:OCP}, we introduce the operators  $\calA : H^1(0,T; \R^n) \to L^2(0,T; \R^n) \times \R^n$ and $\calB: L^2(0,T; \R^m) \to L^2(0,T; \R^n) \times \R^n$
where 
\begin{equation}\label{eq:defnA}
    \calA x = \begin{smallbmatrix}
        \tfrac{\mathrm{d}}{\mathrm{d}t}x-Ax \\ x(0)
    \end{smallbmatrix}, \qquad \calB u = \begin{smallbmatrix}
        Bu \vphantom{\tfrac{\mathrm{d}}{\mathrm{d}t}x-Ax} \\ 0 \vphantom{x(0)}
    \end{smallbmatrix} .%
\end{equation}
Consequently, \eqref{eq:fin_dim_cs} is abstractly written as 
\begin{equation}\label{eq:abstracode}
    \calA x - \calB u = \begin{smallbmatrix}
         0 \\ x_0
     \end{smallbmatrix}.
\end{equation}
To include state feedbacks into the abstract operator setting, we define the bounded multiplication operator $\calK : L^2(0,T; \R^n) \to L^2(0,T; \R^m)$ via
\begin{equation}\label{eq:multop}
    (\calK x)(t) = Kx(t)
\end{equation} 
for a given state feedback matrix $K\in  \R^{n \times m}$. Then, the closed-loop dynamics associated with the static feedback $u(t) = Kx(t) + v(t)$ for all $t\in [0,T]$ reads
 \begin{equation}\label{eq:fin_dim_closed_loop}
    \dot x(t) = (A + BK)x(t) + Bv(t)  , \qquad x(0) = x_0.
\end{equation}
with unique solution given by the variations-of-constants formula
\begin{equation*}%
x(t) = e^{t(A+BK)}x_0 + \int_0^t e^{(t-s)(A+BK)} B v(s)\,\mathrm{d}s.
\quad t \in [0,T].
\end{equation*}
Hence, if one can suitably shape the spectrum of $A$ with the feedback matrix $K$ such that $\omega_0(A+BK) < 0$, then \eqref{eq:Gronwall} holds also with any $\omega_0(A+BK) < \omega < 0$. This setting is termed \emph{exponential stabilizability} in the realm of mathematical system theory \cite[Chapter 5]{curtain2020introduction}, as the system \eqref{eq:fin_dim_closed_loop} is subject to an exponentially stable matrix.

Correspondingly, the feedbacked version of \eqref{eq:abstracode} may be obtained by replacing $u$ with $v$ and $\calA$ with
\begin{equation}\label{eq:feedbackedcalA}
    \tcalA x = (\calA - \calB \calK) x = \begin{smallbmatrix}
        \tfrac{\mathrm{d}}{\mathrm{d}t}x-(A+BK) x \\ x(0)
    \end{smallbmatrix}
\end{equation}

The following result yields a bound on the operator norm in dependence of the stability parameter $\omega$. The proof is presented in Appendix \ref{subsec:app1}.
\begin{proposition}\label{prop:ode}
    For any $A\in \R^{n\times n}$, the operator $\calA$ as defined in \eqref{eq:defnA} is bounded and boundedly invertible. In particular,
    \begin{align*}
        \|\calA^{-1}\|_{L^2(0,T;\R^n)\times \R^n \to H^1(0,T;\R^n)} \leq \begin{cases}
            \left( \tfrac{M^2}{\omega} \left(1 + 2 \|A\|^2\right)  (e^{2T\omega}-1)+2\right)^{1/2} & \omega \neq 0\\
            \left(2M^2T \left(1 + 2 \|A\|^2\right) +2\right)^{1/2} & \omega = 0
        \end{cases}
    \end{align*}
    where $(M,\omega)$ are taken from the matrix stability estimate \eqref{eq:Gronwall}.
\end{proposition}

The following result is immediate from Proposition~\ref{prop:ode}.
\begin{corollary}
    Let the control system $(A,B)$ be exponentially stabilizable, that is, there is $K\in \R^{m\times n}$ such that $\omega_0(A+BK)<0$. Then, choosing $\calK:  L^2(0,T; \R^n) \to L^2(0,T; \R^m)$ as in \eqref{eq:multop}, the feedbacked state operator $\widetilde{\calA} \coloneqq \calA - \calB \calK$ as in \eqref{eq:feedbackedcalA} is bounded and boundedly invertible uniformly with respect to the time horizon $T$.
\end{corollary}
\begin{proof}
   By exponential stabilizability of $(A,B)$, for all $\omega > \omega_0$ there exists $M \geq 1$ such that $$
        \left\|e^{t (A+BK)} \right\|_{\R^n\to\R^n} \leq M e^{\omega t} $$ for all $t\geq 0$.  We choose $\omega = \omega_0/2<0$. Then, it is immediate from Proposition \ref{prop:ode} that
    \begin{align*}
        \|\calA^{-1}\|_{L^2(0,T;\R^n)\times \R^n \to H^1(0,T;\R^n)} &\leq
           \left( \tfrac{M^2}{\omega} \left(1 + 2 \|A+BK\|^2\right)  (e^{2T\omega}-1)+2\right)^{1/2} \\
            &\leq \left( \tfrac{2M^2}{-\omega_0} \left(1 + 2 \|A+BK\|^2\right) +2\right)^{1/2} .
    \end{align*}
\end{proof}

\subsection{Elliptic PDEs}\label{subsec:elliptic}
In this part, we study an application to problems governed by elliptic partial differential equations. To this end, let $\Omega \subset \mathbb{R}^n$ be a bounded Lipschitz domain. We consider the following distributed control problem governed by a linear elliptic convection--diffusion--reaction equation:
\begin{equation}\label{eq:elliptic}
    - \operatorname{div}(\kappa(\omega)\nabla x(\omega))
    + b(\omega)^\top \nabla x(\omega)
    + c(\omega) x(\omega)
    = \chi_{\Omega_c}(\omega) u(\omega)
    \qquad \text{for }\omega \in \Omega,
\end{equation}
supplemented with suitable boundary conditions to be specified later. Here, $x:\Omega\to \R$ is the state variable and $u:\Omega \to \R$ denotes the distributed control acting on a subdomain which enters the equation via a characteristic function $\chi_{\Omega_c}$ of the control domain $\Omega_c\subset \Omega$. Throughout, we assume for the coefficients the standard regularity assumptions, that is  $\kappa\in L^\infty(\Omega;\R^{n\times n}),b\in L^\infty(\Omega;\R^n), c\in L^\infty(\Omega)$. Further, we assume that the matrix field is uniformly elliptic, that is, there is $\underline{\kappa}>0$ such that $\|v\|_{\R^n}^2 \leq v^\top \kappa(\omega)v$ for all $v\in \R^n$ and $\omega \in \Omega$. %

Let $\calX\subset H^1(\Omega)$, depending on the imposed boundary conditions. For instance, in the case of homogeneous Dirichlet boundary conditions one chooses $\calX = H_0^1(\Omega)$
whereas for Neumann or Robin boundary conditions one typically takes $\calX = H^1(\Omega)$. Define the bilinear form $a : \calX \times \calX \to \mathbb{R}$ by
\begin{align}\label{eq:secondorder_op}
    a(x,y)\coloneqq \langle \kappa\nabla x, \nabla y \rangle_{L^2(\Omega;\mathbb{R}^n)}
    +
      \langle b^\top \nabla x, y \rangle_{L^2(\Omega)}
    +
     c \langle x, y \rangle_{L^2(\Omega)}
    +
    a_{\partial \Omega}(x,y),
\end{align}
where the boundary contribution $a_{\partial \Omega}$ depends on the imposed boundary conditions. For example, $a_{\partial \Omega} \equiv 0$
in the Dirichlet and Neumann cases, whereas for Robin boundary conditions $\partial x = \kappa x$ on $\partial \Omega$ one has $    a_{\partial \Omega}(x,y)
    =
    \int_{\partial \Omega} \kappa x y \, ds$.
    
To ensure well-posedness, we assume throughout that the bilinear form $a$ is continuous and coercive on $\calX$. 
\begin{assumption}
There are constants $m,M>0$ such that
\begin{align*}
    | a(x,y)| \leq M \| x \|_\calX \|y\|_\calX , \qquad a(x,x) \geq m \| x \|_\calX^2 \quad \text{for all } x,y \in \calX.
\end{align*}
\end{assumption}
Our setting covers, for example, the following standard situations:
\begin{enumerate}[label=(\roman*)]
    \item $b=0$, $c=0$ and homogeneous Dirichlet boundary conditions. In this case the coercivity constant is usually given by the Poincaré inequality. 
    \item $b=0$ and $c>0$ and Neumann boundary conditions. Here, coercivity is obtained through the reaction term.
    \item convection--diffusion-type  operators with $c_0 - \frac12 \operatorname{div}(b) \ge 0$, assuming that $b$ admits a weak divergence.\end{enumerate}
Continuity of $a$ ensures the existence of a bounded linear operator $\calA: \calX \to \calX^*$ with 
\begin{equation}\label{eq:opdefinedelliptic}
    \langle \calA x,y \rangle_{\calX^*,\calX}= a(x,y) , \qquad x,y \in \calX.
\end{equation}
Define the operator $\calB: L^2(\Omega) \to \calX^*$ with $\calB u = \calR_{L^2(\Omega)} \chi_{\Omega_c}$. Then the problem \eqref{eq:elliptic} may abstractly be written by 
\begin{equation*}
    \calA x - \calB u=0.
\end{equation*}
The Lax-Milgram Theorem, \cite[Lemma 2.2.1.1.]{grisvard2011elliptic}, yields that $\calA$ is boundedly invertible. Further, one straightforwardly may compute for $x\in \calX,f\in \calX^*$,
\begin{equation}\label{eq:coerc1}
     \calA x = f \quad \Longrightarrow \quad m\|x\|^2_\calX \leq \langle \calA x,x\rangle_{\calX^*,\calX} = \langle f,x\rangle_{\calX^*,\calX} \leq \|f\|_{\calX^*}\|x\|_{\calX}
\end{equation}
and hence
\begin{align}\label{eq:coerc2}
    \| \calA^{-1} \|_{\calX^*\to \calX} \leq \frac1m.
\end{align}
However, %
this also shows that if the coercivity constant $m$ tends to $0$, then the operator norm of the inverse $\mathcal{A}^{-1}$ may become unbounded. Thus, as $m \to 0$, the problem becomes increasingly ill-conditioned, which may lead to significant difficulties in numerical approximations. Hence, one may introduce a feedback operator to improve the coercivity constant of the feedbacked state operator, that is, to obtain
\begin{align*}
   \langle (\calA - \calB \calK) x,x\rangle_{\calX^*,\calX} \geq \widetilde m \|x\|^2_{\calX} \quad \forall x\in \calX \quad \mathrm{with} \quad \widetilde m \gg m.
\end{align*}
such that the feedbacked operator satisfies
\begin{align*}
    \| \tcalA^{-1} \|_{\calX^*\to \calX} =  \| (\calA - \calB \calK)^{-1} \|_{\calX^*\to \calX} \leq \frac{1}{\widetilde m} \ll \frac1m.
\end{align*}
We now illustrate this approach for a concrete example.

\medskip

 \textbf{Distributed control of Neumann problem}.
We consider the Poisson problem with homogeneous Neumann boundary condition
\begin{align}\label{eq:Neumann}
\begin{aligned}
-\Delta x + \gamma x & = u & \qquad \text{on } & \Omega, \\
\partial_N x         & = 0 & \qquad \text{on } & \partial\Omega.
\end{aligned}
\end{align}
with $\gamma > 0$ and where $\partial_N$ denotes the normal derivative of $x$. Note that this corresponds to~\eqref{eq:elliptic} by choosing $\kappa \equiv I, b\equiv 0, c \equiv \gamma$ and $\Omega_c = \Omega$. The bilinear form associated with \eqref{eq:Neumann} is given by $a: H^1(\Omega) \times H^1(\Omega) \to \R$ with 
\begin{equation*}
     a(x,y) =  \langle \nabla x, \nabla y \rangle_{L^2(\Omega;\mathbb{R}^n)} + \gamma \langle x,y \rangle_{L^2(\Omega)}
\end{equation*}
We note that 
\begin{align}\label{eq:coerc_ex}
    a(x,x) = \|\nabla x\|^2_{L^2(\Omega;\R^d)} 
     + \gamma\|x\|^2_{L^2(\Omega)}
     \ge \min\{\gamma,1\}\|x\|^2_{H^1(\Omega)}.
\end{align}
For constant functions $z\in H^1(\Omega)$, we get $a(z,z) = \gamma \|x\|^2_{L^2(\Omega)} = \gamma \|x\|^2_{H^1(\Omega)}$ such that the coercivity constant in \eqref{eq:coerc_ex} is optimal in the case $\gamma < 1$ in the sense that it can not be chosen larger. In view of \eqref{eq:coerc1} and \eqref{eq:coerc2}, this yields that the associated operator $\calA : H^1(\Omega) \to H^1(\Omega)^*$ defined via \eqref{eq:opdefinedelliptic} satisfies the solution operator norm
 \begin{equation}\label{eq:elliptic_original}
      \| \calA^{-1} \|_{H^{1}(\Omega)^*\to H^{1}(\Omega)} = \frac{1}{\gamma}, \qquad \mathrm{if} \ \gamma < 1.
 \end{equation}
 Hence, if $\gamma>0$ is small, the solution operator norm becomes very large; in the limit $\gamma \to 0$, $\calA$ even lacks to be invertible due to a kernel consisting of the subspace of constant functions. 
We now show how to improve this solution operator norm by means of a feedback transformation. To this end, let $\calK : L^2(\Omega) \to L^2(\Omega)$ with 
\begin{equation}\label{eq:feedback_elliptic}
    \calK x = -(-\gamma  + 1)x = (\gamma-1) x
\end{equation}
Substituting $u=\calK x +v$, the control problem \eqref{eq:Neumann} is equivalently rewritten by 
\begin{align}\label{eq:ellipticfeedbacked}
    0 = -\Delta x +\gamma x - u = -\Delta x  + \gamma x - \calK x -v = -\Delta x + x -v \qquad \text{on } \Omega .
\end{align}
This new control problem is now driven by $\widetilde \calA: H^1(\Omega) \to H^{1}(\Omega)^*$ with 
\begin{equation*}
    \langle  \widetilde \calA x, y\rangle_{H^1(\Omega)^*,H^1(\Omega)} = \langle (\calA - \calB \calK)x, y\rangle_{H^1(\Omega)^*,H^1(\Omega)} = \langle \nabla x, \nabla y \rangle_{L^2(\Omega;\mathbb{R}^n)} + \langle x,y \rangle_{L^2(\Omega)}, 
\end{equation*}
which clearly satisfies  $\langle \tcalA x,x\rangle_{H^1(\Omega)^*,H^1(\Omega)} = \|x\|^2_{H^1(\Omega)}$
such that
\begin{equation}\label{eq:elliptic_feedbacked}
    \|\tcalA^{-1}\|_{H^1(\Omega)^*\to H^1(\Omega)} = 1,
\end{equation}
that is, the feedbacked solution operator is bounded independently of $\gamma$. Even the limiting case $\gamma = 0$, in which $\calA$ is surjective but not injective may thus be treated using a gradient or PPCG method, as we only require the feedbacked control-to-state map to solve \eqref{eq:OCP2}.

To apply Propositions~\ref{prop:grad_cond_feedbacked} and~\ref{prop:grad_ppcg_feedbacked}, we remain to verify Assumption~\ref{ass:feedback}. Assumption~\ref{ass:feedback}\ref{ass:iso}, i.e., continuous invertibility of $\tcalA$ follows directly from \eqref{eq:elliptic_feedbacked}.
Assumption~\ref{ass:feedback}\ref{ass:smallness} is clearly satisfied with $\delta_\calK = (1-\gamma)>0$. Hence, both Proposition~\ref{prop:grad_cond_feedbacked} and \ref{prop:grad_ppcg_feedbacked} are applicable.

\begin{remark}
We briefly discuss the case in which we only have control acting on a part of the domain $\Omega_c\subset \Omega$. Here, the feedbacked operator is then (slightly abusing the notation and denoting the operators in strong form) $\calA - \calB \calK = -\Delta + \gamma I + k\chi_{\Omega_c}$,
where $\chi_{\Omega_c}$ is the characteristic function of the control set and $k>0$ characterizes the feedback $\calK = -kI$. If this is a set of non-zero Lebesgue measure, then the operator above is still uniformly (in $\gamma$) elliptic due to the generalized Poincaré inequality~\cite{Tro10}. A similar comment applies for boundary control.
\end{remark}

In Figure~\ref{fig:grad:elliptic}, we show the results for the condensed conjugate gradient method as outlined and analyzed in Subsection~\ref{subsec:gradfeedbacked}. We solve the problem \eqref{eq:OCP} subject to the elliptic PDE~\eqref{eq:Neumann} on the unit square $[-1,1]^2\subset \R^2$, $\calY = L^2(\Omega)$, $\calC = I$, $y_\mathrm{ref}(\omega_1,\omega_2) = \sin(2\pi\omega_1) + \sin(2\pi \omega_2)$ for varying regularization parameters $\alpha$ in the cost functional and $\gamma > 0$ in the PDE. As a comparison, we solve the feedbacked counterpart~\eqref{eq:OCP2} using the feedback as defined in~\eqref{eq:feedback_elliptic}. Combining the deduced condition number estimates of Proposition~\ref{prop:gradient:condition} for the original problem and Proposition~\ref{prop:grad_ppcg_feedbacked} for its feedbacked version and the solution operator estimates \eqref{eq:elliptic_original} and \eqref{eq:elliptic_feedbacked}, we see that the condition number for the original problem deteriorates for $\gamma \to 0$, while it remains uniformly (in $\gamma$) bounded for the feedbacked problem. This may clearly be observed in Figure~\ref{fig:grad:elliptic}. When decreasing the parameter $\gamma$ (left to right), we observe that the convergence behavior of the original problem becomes worse and worse (for both choices of the regularization parameter depicted in the upper and lower row). In contrast, the convergence behavior of the feedbacked problem is visibly independent of $\gamma$. This clearly reflects the uniform bound on the solution operator \eqref{eq:elliptic_feedbacked}.
\begin{figure}[htb]
    \centering
\scalebox{0.8}{%
\begin{tikzpicture}

\begin{axis}[
    width=0.35\textwidth,
    height=0.3\textwidth,
    xlabel={CG iterations},
    ylabel={Relative error},
    xmin=1, xmax=9,
    xtick={1,3,5,7,9},
    ymode=log,
    ytickten={0,-2,-4,-6,-8,-10,-12,-14},
    grid=both,
    legend pos=north east,
    thick,
    title={$(\alpha, \gamma) = (0.001,0.1)$}
]
\addplot[mark=o, blue] coordinates {
    (1, 0.726838885538329)
    (3, 0.009929677474466606)
    (5, 0.0001050002019194221)
    (7, 6.379370368123632e-07)
    (9, 5.49298501624494e-11)
};
\addlegendentry{Original}
\addplot[mark=square*, red!70!black] coordinates {
    (1, 0.7061271740160248)
    (3, 0.0033930805278799335)
    (5, 5.718089076825183e-05)
    (7, 4.4696718182320363e-10)
    (9, 5.49298501624494e-11)
};
\addlegendentry{Feedbacked}
\end{axis}

\end{tikzpicture}
}
\scalebox{0.8}{%
\begin{tikzpicture}

\begin{axis}[
    width=0.35\textwidth,
    height=0.3\textwidth,
    xlabel={CG iterations},
    xmin=1, xmax=9,
    xtick={1,3,5,7,9},
    ymode=log,
    ytickten={0,-2,-4,-6,-8,-10,-12,-14},
    grid=both,
    legend pos=north east,
    thick,
    title={$(\alpha, \gamma) = (0.001,0.001)$}
]
\addplot[mark=o, blue] coordinates {
    (1, 0.7314318725637456)
    (3, 0.11358871083364624)
    (5, 0.0045227461133589145)
    (7, 0.0001063256463409562)
    (9, 6.459714474224866e-07)
};
\addplot[mark=square*, red!70!black] coordinates {
    (1, 0.7126181353324369)
    (3, 0.0035220926235417418)
    (5, 6.146756652743834e-05)
    (7, 5.713431812163206e-10)
    (9, 6.610434908164832e-11)
};
\end{axis}

\end{tikzpicture}
}
\scalebox{0.8}{%
\begin{tikzpicture}

\begin{axis}[
    width=0.35\textwidth,
    height=0.3\textwidth,
    xlabel={CG iterations},
    xmin=1, xmax=9,
    xtick={1,3,5,7,9},
    ymode=log,
    ytickten={0,-2,-4,-6,-8,-10,-12,-14},
    grid=both,
    legend pos=north east,
    thick,
    title={$(\alpha, \gamma) = (0.001,10^{-5})$}
]
\addplot[mark=o, blue] coordinates {
    (1, 0.7314781416694468)
    (3, 0.11360212992062516)
    (5, 0.004523264484656682)
    (7, 0.00010985271952152105)
    (9, 0.000106281447640779)
};
\addplot[mark=square*, red!70!black] coordinates {
    (1, 0.7126829817534912)
    (3, 0.0035234004375952416)
    (5, 6.151171794060015e-05)
    (7, 5.765141172326781e-10)
    (9, 4.60541640890441e-11)
};
\end{axis}

\end{tikzpicture}
}
\\
\scalebox{0.8}{%
\begin{tikzpicture}

\begin{axis}[
    width=0.35\textwidth,
    height=0.3\textwidth,
    xlabel={CG iterations},
    ylabel={Relative error},
    xmin=10, xmax=30,
    xtick={10,20,30},
    ymode=log,
    ytickten={0,-2,-4,-6,-8,-10,-12,-14},
    grid=both,
    legend pos=north east,
    thick,
    title={$(\alpha, \gamma) = (10^{-5},0.1)$}
]
\addplot[mark=o, blue] coordinates {
    (10, 0.011733192233152737)
    (20, 3.274645659341223e-05)
    (30, 8.762544894739712e-08)
};
\addplot[mark=square*, red!70!black] coordinates {
    (10, 0.0056951952936283535)
    (20, 5.507430247454323e-06)
    (30, 9.7472176661239e-09)
};
\end{axis}

\end{tikzpicture}
}
\scalebox{0.8}{ %
\begin{tikzpicture}

\begin{axis}[
    width=0.35\textwidth,
    height=0.3\textwidth,
    xlabel={CG iterations},
    xmin=10, xmax=30,
    xtick={10,20,30,40},
    ymode=log,
    ytickten={0,-2,-4,-6,-8,-10,-12,-14},
    grid=both,
    legend pos=north east,
    thick,
    title={$(\alpha, \gamma) = (10^{-5},0.001)$}
]
\addplot[mark=o, blue] coordinates {
    (10, 0.024615928043496813)
    (20, 0.00020473285790086252)
    (30, 5.467610166908902e-06)
    (40, 6.595957978664924e-08)
};
\addplot[mark=square*, red!70!black] coordinates {
    (10, 0.00746295101417509)
    (20, 5.676436165620361e-06)
    (30, 1.0483374674945448e-08)
    (40, 7.697565562018415e-11)
};
\end{axis}

\end{tikzpicture}
}
\scalebox{0.8}{%
\begin{tikzpicture}

\begin{axis}[
    width=0.35\textwidth,
    height=0.3\textwidth,
    xlabel={CG iterations},
    xmin=10, xmax=30,
    xtick={10,20,30,40},
    ymode=log,
    ytickten={0,-2,-4,-6,-8,-10,-12,-14},
    grid=both,
    legend pos=north east,
    thick,
    title={$(\alpha, \gamma) = (10^{-5},10^{-5})$}
]
\addplot[mark=o, blue] coordinates {
    (10, 0.10641841984860713)
    (20, 0.004220039417801839)
    (30, 0.00023336971197845894)
    (40, 2.2300893678039672e-05)
};
\addplot[mark=square*, red!70!black] coordinates {
    (10, 0.008633129418152677)
    (20, 5.283177595836857e-06)
    (30, 1.1323333055598219e-08)
    (40, 2.309752168457207e-09)
};
\end{axis}

\end{tikzpicture}
}
\caption{Gradient method for elliptic problem: Relative error over CG iterations for varying regularization parameters and reaction terms.}
\label{fig:grad:elliptic}
\end{figure}
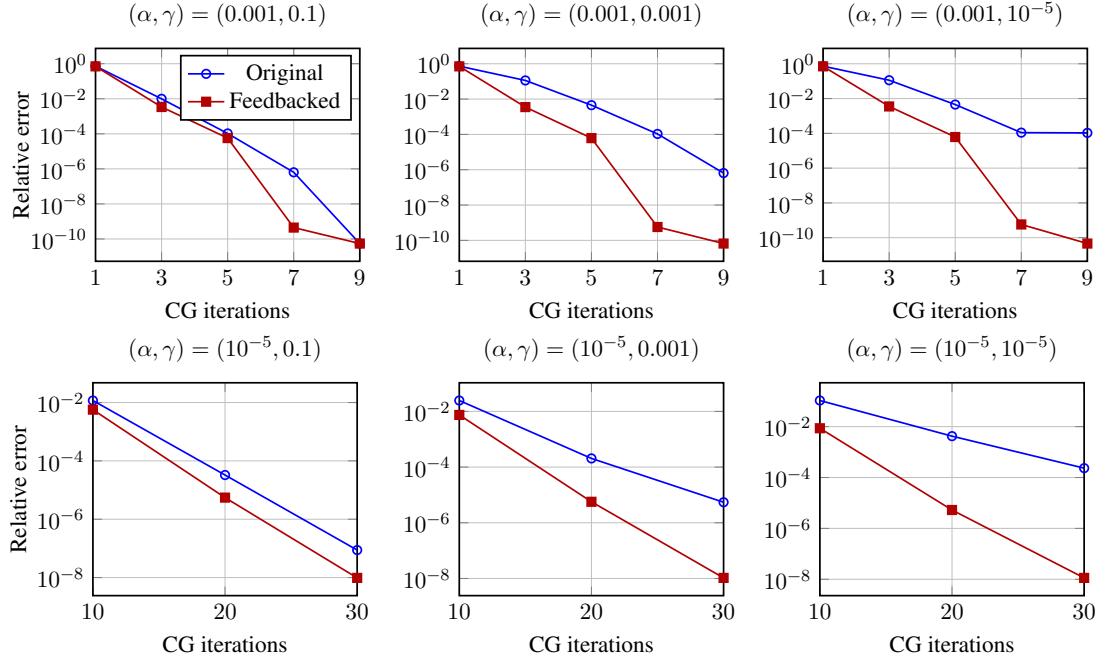

In Figure~\ref{fig:ppcg:elliptic} we show the results for the PPCG method applied to the same problem. We observe the same behavior as for the gradient method: For decreasing parameters $\gamma > 0$, the convergence of the PPCG method deteriorates, while for the feedbacked problem it remains unchanged. This again reflects the theoretical findings corresponding to the condition number bounds of Propositions~\ref{prop:ppcg:condition} and~\ref{prop:grad_ppcg_feedbacked} together with the solution operator norms~\eqref{eq:elliptic_original} and~\eqref{eq:elliptic_feedbacked}.
\begin{figure}[htb]
    \centering
\scalebox{0.8}{%
\begin{tikzpicture}

\begin{axis}[
    width=0.35\textwidth,
    height=0.3\textwidth,
    xlabel={CG iterations},
    xmin=1, xmax=80,
    xtick={10,20,30,40,50,60,70,80},
    ymode=log,
    ytickten={0,-2,-4,-6,-8,-10,-12,-14},
    grid=both,
    legend pos=north east,
    thick,
    title={$(\alpha, \gamma) = (10^{-6},10^{-2})$}
]
\addplot[mark=o, blue] coordinates {
    (2, 0.04941758191711827)
    (4, 0.041436358256348296)
    (8, 0.02568596087456935)
    (15, 0.004946684162894679)
    (25, 0.0012084860081291373)
    (50, 7.129802433804703e-06)
    (75, 8.337032976813433e-08)
    (100, 1.7803515706513951e-09)
};
\addplot[mark=square*, red!70!black] coordinates {
    (2, 0.04930440410061063)
    (4, 0.03722440928886701)
    (8, 0.013157310461967477)
    (15, 0.002412359042185279)
    (25, 0.0001442648043380194)
    (50, 2.1375563926056738e-07)
    (75, 1.7172529261798029e-09)
    (100, 1.0459318892870784e-11)
};
\end{axis}

\end{tikzpicture}
}
\scalebox{0.8}{%
\begin{tikzpicture}

\begin{axis}[
    width=0.35\textwidth,
    height=0.3\textwidth,
    xlabel={CG iterations},
    xmin=1, xmax=80,
    xtick={10,20,30,40,50,60,70,80},
    ymode=log,
    ytickten={0,-2,-4,-6,-8,-10,-12,-14},
    grid=both,
    legend pos=north east,
    thick,
    title={$(\alpha, \gamma) = (10^{-6},0.0001)$}
]
\addplot[mark=o, blue] coordinates {
    (2, 0.05238256884398399)
    (4, 0.04941990765645316)
    (8, 0.03723718015859486)
    (15, 0.014684344455042218)
    (25, 0.004455847969368937)
    (50, 0.00012000886653588281)
    (75, 6.392101194496124e-06)
    (100, 8.31332772275508e-07)
};
\addplot[mark=square*, red!70!black] coordinates {
    (2, 0.0493068287993752)
    (4, 0.03997786067705025)
    (8, 0.014599400851054854)
    (15, 0.0024128033289966895)
    (25, 0.00012137174809026974)
    (50, 1.6894792557717154e-07)
    (75, 1.1546728778916745e-09)
    (100, 9.178088780535181e-12)
};
\end{axis}

\end{tikzpicture}
}
\scalebox{0.8}{%
\begin{tikzpicture}

\begin{axis}[
    width=0.35\textwidth,
    height=0.3\textwidth,
    xlabel={CG iterations},
    xmin=1, xmax=80,
    xtick={10,20,30,40,50,60,70,80},
    ymode=log,
    ytickten={0,-2,-4,-6,-8,-10,-12,-14},
    grid=both,
    legend pos=north east,
    thick,
    title={$(\alpha, \gamma) = (10^{-6},10^{-6})$}
]
\addplot[mark=o, blue] coordinates {
    (2, 0.05238257281489796)
    (4, 0.04941993147227241)
    (8, 0.03723719612965626)
    (15, 0.010736828818260553)
    (25, 0.0050314560491186286)
    (50, 0.000535200602925313)
    (75, 5.8784638994884564e-05)
    (100, 1.1259554690676407e-05)
};
\addplot[mark=square*, red!70!black] coordinates {
    (2, 0.04930685303952718)
    (4, 0.040380852926308254)
    (8, 0.014763042615168049)
    (15, 0.0024128086422607973)
    (25, 0.00012162722684809979)
    (50, 2.2556879783423876e-07)
    (75, 9.453292612667716e-10)
    (100, 8.07171153385791e-12)
};
\end{axis}

\end{tikzpicture}
}
\caption{PPCG method for elliptic problem: Relative error over CG iterations for varying regularization parameters and reaction terms.}
    \label{fig:ppcg:elliptic}
\end{figure}
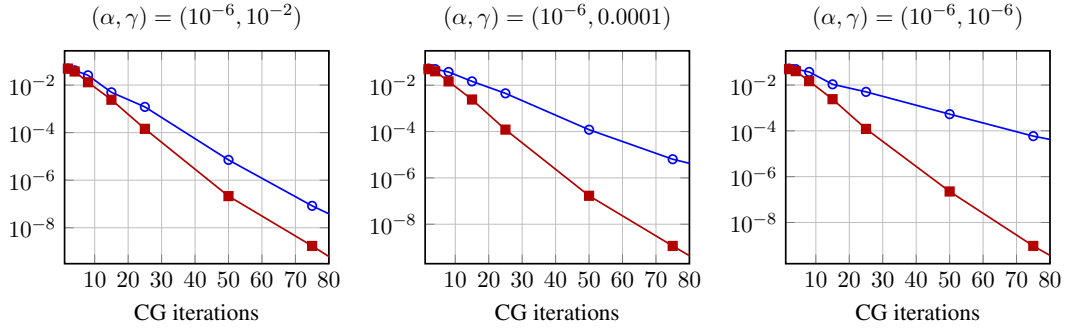

\textbf{Optimization of the condition number \eqref{eq:condnumber}}.
Here, we briefly discuss the extreme case in which $\gamma = 0$. In this case, $\calA$ is not invertible. However, one may easily show that the feedbacked problem is invertible, as $$\sigma_{-\delta I} = \|(\calA - \calB \calK)^{-1}\|_{H^1(\Omega)^*\to H^1(\Omega)} = \max\{\delta^{-1},1\}$$ when using the feedback $\calK = -\delta I$ for some $\delta>0$. Further, in this case, we may minimize the expression \eqref{eq:condnumber} over $\delta$ and obtain $\delta_{\min}
=
\frac{1}{\sqrt{2\alpha}}$ with associated minimal value $\kappa_{\min}=9.$

In Figure~\ref{fig:condopt}, we illustrate the performance of the feedback transformation for different choices of $\alpha$. Therein, in the respective left figures plot the error of the PPCG iteration after 5 iterations over different choices of $\delta$ and the right figures plot the upper bound in \eqref{eq:condnumber}. We clearly observe a non-monotone behavior of the performance with respect to the chosen feedback. Further, the optimal performance matches quite well the optimization of the condition number estimate (depicted with a vertical red line), despite being only an upper bound.
\begin{figure}[htb]
    \centering
    \includegraphics[width=.8\linewidth]{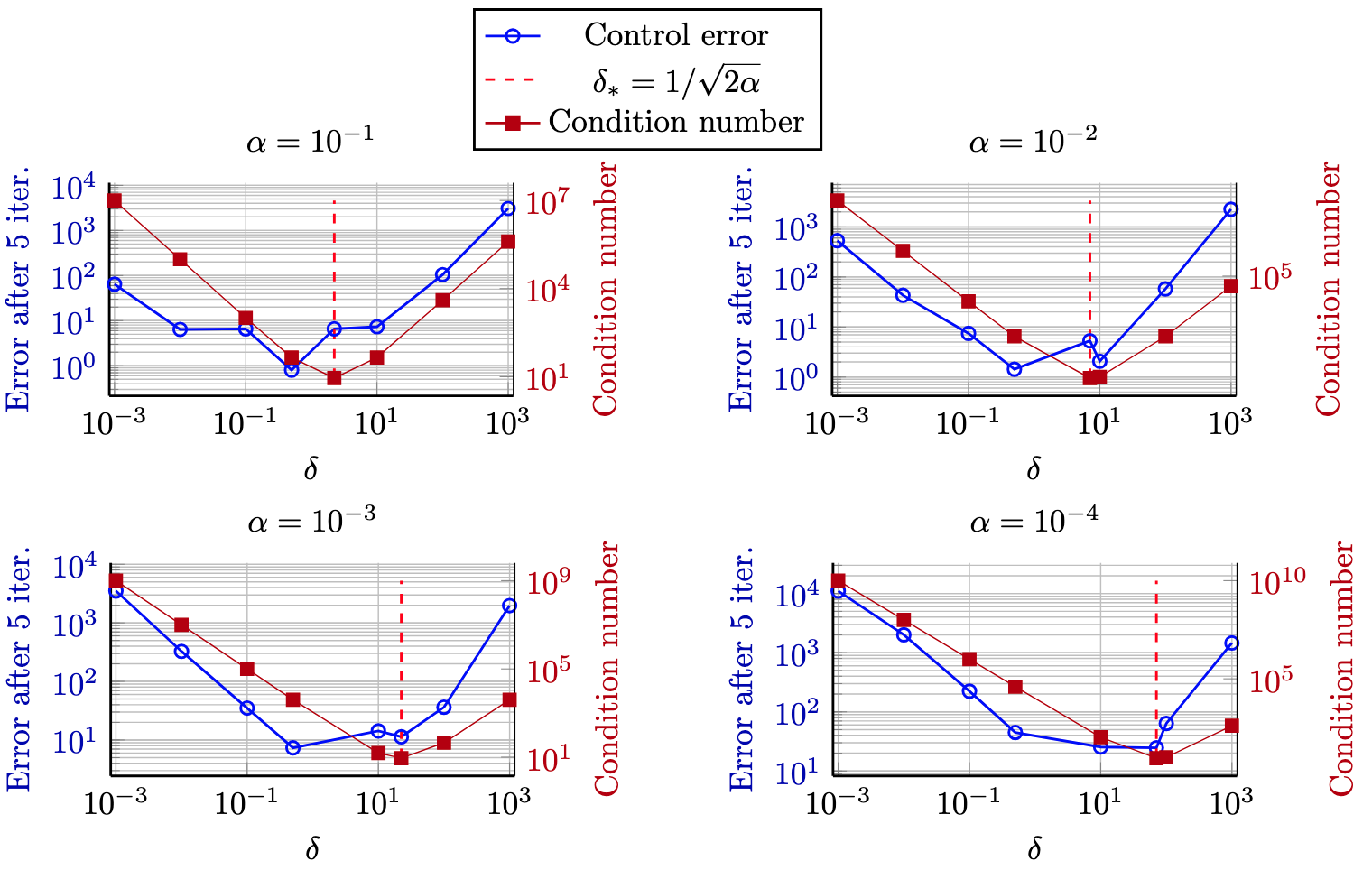}
    \caption{Error after a fixed number of iterations and condition number in dependence of the feedback strength.}
    \label{fig:condopt}
\end{figure}

\subsection{Parabolic PDEs}\label{subsec:parabolic}
In this part, we provide an application to parabolic problems  such as the heat equation, the advection-reaction-diffusion equation, Oseen equations or beam equation with viscous damping.
To this end, let $V$ be a separable and reflexive Banach space that is densely and continuously embedded into a Hilbert space $H$ giving rise to a Gelfand triple $   V\hookrightarrow H \hookrightarrow V^*$, where $\|v\|_H \le \|v\|_V \quad \forall v\in V$. 
We assume that the main operator $A\in L(V,V^*)$ satisfies the generalized ellipticity or G\aa rding inequality %
\begin{align}\label{eq:garding}
\exists c \in \R, m > 0, \beta \ge 0: \quad \langle Av,v\rangle_{V^*,V} + c \|v\|^2_H \geq \|v\|_{m,\beta}^2 := m \|v\|^2_V+\beta\|v\|^2_H.
\end{align}
Observe that we have some freedom in the choice of $m$, $c$, and $\beta$, which will turn out to be useful later on. The use of these weighed norms will allow us to obtain quantitative results on the norm of the solution operator.  

For example, $A$ could be a second-order operator in weak form induced by a bilinear form \eqref{eq:secondorder_op}, possibly with feed-back. Throughout this section, we abbreviate $\|A\|_{V\to V^*}$ by $\|A\|$ and write $\dot x$ for the weak time-derivative of $x$.  As a solution space, we consider the space
\begin{equation*}
\begin{array}{rcl}
      W(0,T) & \coloneqq & \left\lbrace x \in L^2(0,T; V) \, \middle| \, \dot x \in L^2(0,T; V^*) \right\rbrace.
\end{array}
\end{equation*}
Note that we have $W(0,T) \hookrightarrow C([0,T]; H)$ continuously \cite{Wloka1987} such that point evaluations, e.g., to prescribe initial values, may be defined.

A parabolic partial differential equation is then given by the initial value problem
\begin{align}\label{eq:parab}
    \mathrm{Find}\ x\in W(0,T):\ \dot x + Ax = g, \qquad x(0)=x^0,
\end{align}
where $g\in L^2(0,T;V^*)$ is a source term and $x^0\in H$ is an initial datum. To put this into perspective w.r.t.\ the abstract optimal control problem~\eqref{eq:OCP}, let $B\in L(U,V^*)$ be a control operator and $u\in L^2(0,T;U)$ a control function. Then, for given $l \in L^2(0,T;V^*), x^0\in H$, we may set
\begin{align*}
    \mathcal{A} x := \begin{bmatrix}
        \dot x + Ax \\
        x(0)
    \end{bmatrix}, \quad \mathcal{B}u = \begin{bmatrix}
        Bu\\
        0
    \end{bmatrix}, \quad f = \begin{bmatrix}
        g\\
        x^0
    \end{bmatrix}
\end{align*}
such that $\calX=W(0,T)$, $U = L^2(0,T;U)$, $P^*=L^2(0,T;V^*)\times H$. It is well-known by parabolic theory \cite{Wloka1987,Zeidler1990,Tro10}, that under the above assumptions and for given $u\in L^2(0,T;U)$, the problem  $\mathcal{A}x + \mathcal{B}u = f$ has a unique solution $x\in W(0,T)$. For a concise proof that $\mathcal{A}$ is indeed an isomorphism, we refer to \cite{schiela2013concise}. In this work, we are particularly interested in norms of the solution operator and suitably feedbacked counterparts. To alleviate notation, we define the following scaled norms
 \[
   \|x\|_{L^2_{c,m,\beta}} := \|e^{-c t}\|x(t)\|_{m,\beta}\|_{L^2(0,T)},  \qquad \|f\|_{L^2_{c,m,\beta},*}:= \|e^{-c t}\|f(t)\|_{m,\beta,*}\|_{L^2(0,T)},
 \]
 where $\|\cdot\|_{m,\beta}$ is defined in \eqref{eq:garding} and $\|\ell\|_{m,\beta,*} := \sup_{\|v\|_{m,\beta}\le 1} |\ell(v)|$.

\begin{proposition}\label{prop:bound_parab}
  Suppose that \eqref{eq:garding} holds and $\calA x=(g,x^0)$ for data $g\in L^2(0,T;V^*)$ and $x^0 \in H$. Then
  \begin{subequations}\label{eq:paraest}
  \begin{align}\label{eq:paraest1}
  \max_{t\in [0,T]} (e^{-ct} \|x(t)\|_H)+\|x\|^2_{L^2_{c,m,\beta}} \!\!&\le \|g\|^2_{L^2_{c,m,\beta},*}+\|x^0\|_H^2\\
\|e^{-ct}\dot x\|_{L^2(0,T;V^*)}&\le \left(\tfrac{\|A\|}{m}+1\right)\|e^{-c\cdot} g\|_{L^2(0,T;V^*)}+\tfrac{\|A\|}{\sqrt{m}}\|x^0\|_H. \label{eq:paraest2}
\end{align}
\end{subequations}
If \eqref{eq:garding} holds with $c\leq 0$ and hence with $c=0$, that is, the main operator $A$ is elliptic and the equation is exponentially stable, there is a constant $C = C(\|A\|,m)$ such that 
\begin{align*}
    \|\calA^{-1}\|_{L^2(0,T;V^*) \times H \to W(0,T)} \leq C.
\end{align*}
\end{proposition}
\begin{proof}
  Due to $e^{-c t} \dot x = \frac{\partial}{\partial t}\left(e^{-c t} x\right) +c e^{-c t} x,$ the state $x\in W(0,T)$ solves \eqref{eq:parab} if and only if the scaled state $\tilde x = e^{-c t} x\in W(0,T)$ solves
\begin{align}\label{eq:modified_parabolic}
    \dot{\tilde x} + (A+c I)\tilde x = e^{-c t} g, \qquad \tilde x(0)=x^0.
\end{align}
Moreover, we have that the governing operator of this equation satisfies
\begin{align*}
    \langle (A + c I) \tilde x,\tilde x\rangle_{V^*,V} = \langle A\tilde x,\tilde x\rangle_{V^*,V} + c \|\tilde x\|^2_H \geq \|\tilde x\|^2_{m,\beta}
\end{align*}
due to \eqref{eq:garding}. Testing \eqref{eq:modified_parabolic} with $(\tilde x,\tilde x(0))$ and using the formula of integration by parts
$\langle \dot \tilde x ,\tilde x \rangle_{L^2(0,T;V^*),L^2(0,T;V)} = \frac12(\|\tilde x(T)\|_H^2 - \|\tilde x(0)\|^2_H)$ we get
    \begin{align*}
        \tfrac12 \|\tilde x(T)\|^2_H + \|\tilde x\|_{L^2_{0,m,\beta}}^2 &\leq \tfrac12 \|x^0\|^2_H + \|e^{-c t} g\|_{L^2_{0,m,\beta},*} \|\tilde x\|_{L^2_{0,m,\beta}}\\
        &\leq \tfrac12 \|x^0\|^2_H + \tfrac{1}{2} \|g\|^2_{{L^2_{c,m,\beta}},*} + \tfrac{1}{2} \|\tilde x\|^2_{L^2_{0,m,\beta}}.
    \end{align*}
 In terms of original variables and scaled norms, this yields
 \[
   \|e^{-c T}x(T)\|_H^2+\|x\|^2_{L^2_{c,m,\beta}}\le \|x^0\|_H^2+\|g\|^2_{L^2_{c,m,\beta},*}.
 \]
 Since the choice of $T$ is arbitrary, we finally obtain \eqref{eq:paraest1}.
 
 By bootstrapping with the equation $\dot x = -Ax+g$ we finally get:
 \begin{align*}
 \|e^{-ct}\dot x\|_{L^2(0,T;V^*)} &\le \|A\|\|e^{-ct} x\|_{L^2(0,T;V)}+\|e^{-ct} g\|_{L^2(0,T;V^*)}\\
 &\le \tfrac{\|A\|}{\sqrt{m}}\sqrt{\|g\|^2_{L^2_{c,m,\beta}}+\|x^0\|^2_H}+\|e^{-ct} g\|_{L^2(0,T;V^*)}\\
 &\le \left(\tfrac{\|A\|}{m}+1\right)\|e^{-ct} g\|_{L^2(0,T;V^*)}+\tfrac{\|A\|}{\sqrt{m}}\|x^0\|_H
 \end{align*}
 and thus \eqref{eq:paraest2}. The last inequality follows by setting $c=0$ in the above estimates, using $\|x\|_{L^2_{0,m,\beta}} \geq \|x\|_{L^2_{0,m,0}} = m \|x\|_{L^2(0,T;V)}$ and adding both \eqref{eq:paraest1} and \eqref{eq:paraest2}. 
\end{proof}
We briefly comment on the consequence of the estimate \eqref{eq:paraest1} that directly implies an estimate on the standard norm due to
\begin{align*}
 \|e^{-ct}x\|_{L_2(0,T;V)} &\le \tfrac{1}{\sqrt{m}}\|x\|_{L^2_{c,m,\beta}}, \quad 
 \|e^{-ct}x\|_{_{L_2(0,T;H)}} \le \tfrac{1}{\sqrt{m+\beta}} \|x\|_{L^2_{c,m,\beta}} \\
 \|g\|_{L^2_{c,m,\beta},*} &\le \tfrac{1}{\sqrt{m}}\|e^{-ct}g\|_{L_2(0,T;V^*)}, \quad \|g\|_{L^2_{c,m,\beta},*} \le \tfrac{1}{\sqrt{m+\beta}} \|e^{-ct}g\|_{L_2(0,T;H)} 
\end{align*}
for $x\in L^2(0,T;V)$ and $f\in L^2(0,T;H)$.

Let us consider now the case of distributed control and observation, i.e., $\calU = L^2(0,T;H)$ and $B\in L(H,H)$. 
As an immediate result, we obtain the following.
\begin{corollary}
  Suppose that \eqref{eq:garding} holds. Then
   \begin{align}
       \sigma_0= \|\calC\mathcal{A}^{-1}\calB\|_{L^2(0,T;H) \to L^2(0,T;H)} &\leq \tfrac{e^{|c|T}}{m+\beta}
    \end{align}
\end{corollary}
\begin{proof}
 This follows from
$
   \|e^{-c t}\|_\infty^{-1}\|e^{-ct}v\|_{L^2} \le \|v\|_{L^2} \le \|e^{c t}\|_\infty\|e^{-ct} v\|_{L^2}
 $
 and our previous results, taking into account that $e^{\pm c t}$ assumes its maximum at a boundary point of the interval $[0,T]$.
\end{proof}
 Next, we provide a concrete example to illustrate the results.

\medskip

 \textbf{Heat equation with Neumann boundary conditions. } We consider the very same setup as in \eqref{eq:Neumann}, but now allow for $\gamma \in \R$, that is, also negative values. Here we have that \eqref{eq:garding} holds for $A=-\Delta + \gamma I$ (associated with $\calA$) with $c=0$ and $m=\min\{\gamma,1\}$ if $\gamma >0$ and with $c= - \gamma + 1>0$ and $m=1$ for $\gamma \leq 0$. In the former case, the uncontrolled system is exponentially stable; in the latter it is stable (if $\gamma = 0$) or exponentially unstable (if $\gamma <0$). For the governing operator of $\tcalA$, i.e. $A-BK$, \eqref{eq:garding} always holds with $c=0$ by the choice of the feedback due to \eqref{eq:ellipticfeedbacked}.

In Figure~\ref{fig:grad:parabolic}, we illustrate the results for the condensed gradient method. For $\gamma>0$ (left), the problem is exponentially stable. However, for $\gamma < 0$, the (uncontrolled) problem corresponding to $\calA$ is exponentially unstable. Hence, when increasing the horizon $T=10$ to $T=30$ (from top to bottom), the convergence behavior significantly worsens for the original problem. This reflects the solution operator norm of Proposition~\ref{prop:bound_parab} depending exponentially on the time horizon. and its impact on the condition number in Proposition~\ref{prop:gradient:condition}. For the feedbacked problem, the convergence behavior again is robust w.r.t.\ an increased time horizon. Hereby we note, however, that the iteration numbers increase (in contrast to the elliptic case of Subsection~\ref{subsec:elliptic}), as an increased horizon leads to a higher number of time discretization points.

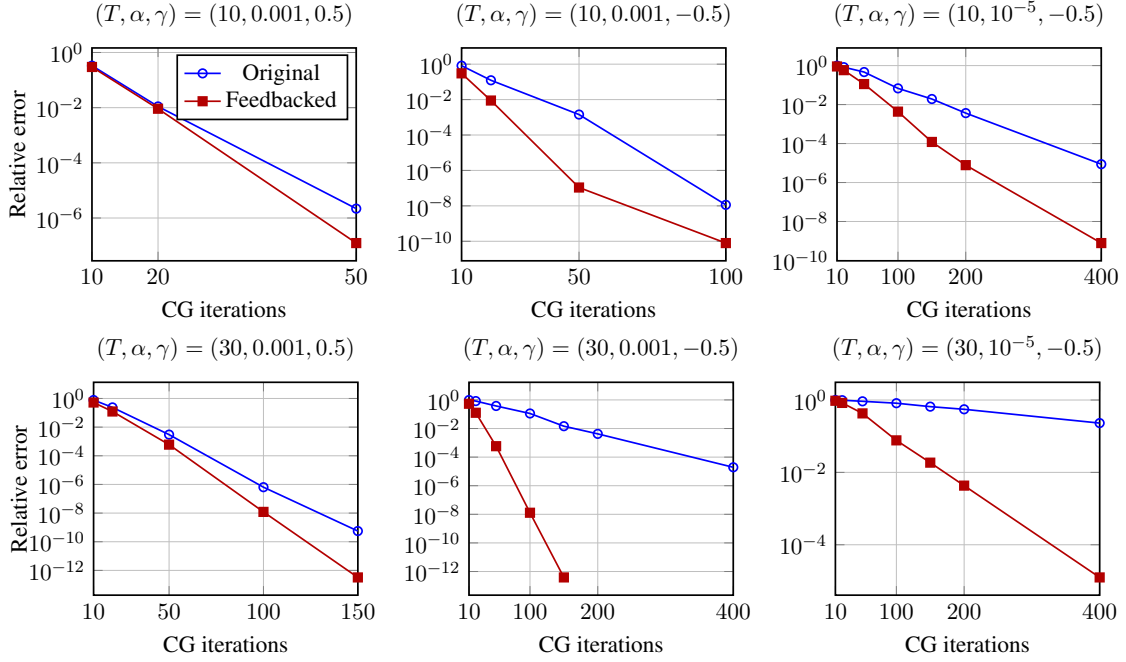
\begin{figure}[htb]
    \centering
\scalebox{0.8}{%
\begin{tikzpicture}

\begin{axis}[
    width=0.35\textwidth,
    height=0.3\textwidth,
    xlabel={CG iterations},
    ylabel={Relative error},
    xmin=10, xmax=50,
    xtick={10,20,50},
    ymode=log,
    ytickten={0,-2,-4,-6,-8,-10,-12,-14},
    grid=both,
    legend pos=north east,
    thick,
    title={$(T,\alpha, \gamma) = (10,0.001,0.5)$}
]
\addplot[mark=o, blue] coordinates {
    (10, 0.33243351519480374)
    (20, 0.011060359908102639)
    (50, 2.1870946142991758e-06)
};
\addlegendentry{Original}
\addplot[mark=square*, red!70!black] coordinates {
    (10, 0.2982651569666979)
    (20, 0.009143035468781433)
    (50, 1.2342818969637377e-07)
};
\addlegendentry{Feedbacked}
\end{axis}

\end{tikzpicture}
}
\scalebox{0.8}{%
\begin{tikzpicture}

\begin{axis}[
    width=0.35\textwidth,
    height=0.3\textwidth,
    xlabel={CG iterations},
    xmin=10, xmax=100,
    xtick={10,50,100},
    ymode=log,
    ytickten={0,-2,-4,-6,-8,-10,-12,-14},
    grid=both,
    legend pos=north east,
    thick,
    title={$(T,\alpha, \gamma) = (10,0.001,-0.5)$}
]
\addplot[mark=o, blue] coordinates {
    (10, 0.8204462370764743)
    (20, 0.12301865727156086)
    (50, 0.0014286185413296363)
    (100, 1.1386718166156318e-08)
};
\addplot[mark=square*, red!70!black] coordinates {
    (10, 0.30263638912738344)
    (20, 0.00884505808847306)
    (50, 1.0855608805835202e-07)
    (100, 7.924603262603529e-11)
};
\end{axis}

\end{tikzpicture}
}
\scalebox{0.8}{%
\begin{tikzpicture}

\begin{axis}[
    width=0.35\textwidth,
    height=0.3\textwidth,
    xlabel={CG iterations},
    xmin=10, xmax=400,
    xtick={10,100,200,400},
    ymode=log,
    ytickten={0,-2,-4,-6,-8,-10,-12,-14},
    grid=both,
    legend pos=north east,
    thick,
    title={$(T,\alpha, \gamma) = (10,10^{-5},-0.5)$}
]
\addplot[mark=o, blue] coordinates {
    (10, 0.9877133452006225)
    (20, 0.8279236834352088)
    (50, 0.46551600435916024)
    (100, 0.06775411227211497)
    (150, 0.019270941447641795)
    (200, 0.003641796041956418)
    (400, 8.727772216441561e-06)
};
\addplot[mark=square*, red!70!black] coordinates {
    (10, 0.9023312355361782)
    (20, 0.5817912329622617)
    (50, 0.11266328361844051)
    (100, 0.004336185203693907)
    (150, 0.00012068332700066538)
    (200, 7.778649882825784e-06)
    (400, 7.815852871387957e-10)
};
\end{axis}

\end{tikzpicture}
}\\
\scalebox{0.8}{%
\begin{tikzpicture}

\begin{axis}[
    width=0.35\textwidth,
    height=0.3\textwidth,
    xlabel={CG iterations},
    ylabel={Relative error},
    xmin=10, xmax=150,
    xtick={10,50,100,150},
    ymode=log,
    ytickten={0,-2,-4,-6,-8,-10,-12,-14},
    grid=both,
    legend pos=north east,
    thick,
    title={$(T,\alpha, \gamma) = (30,0.001,0.5)$}
]
\addplot[mark=o, blue] coordinates {
    (10, 0.7776617464884915)
    (20, 0.23795658730204677)
    (50, 0.002999718654450941)
    (100, 6.417287392719695e-07)
    (150, 5.581024844764491e-10)
};
\addplot[mark=square*, red!70!black] coordinates {
    (10, 0.5217966012537908)
    (20, 0.12333015314925289)
    (50, 0.0005991653039803581)
    (100, 1.2110420118253084e-08)
    (150, 3.2327179809492613e-13)
};
\end{axis}

\end{tikzpicture}
}
\scalebox{0.8}{%
\begin{tikzpicture}

\begin{axis}[
    width=0.35\textwidth,
    height=0.3\textwidth,
    xlabel={CG iterations},
    xmin=10, xmax=400,
    xtick={10,100,200,400},
    ymode=log,
    ytickten={0,-2,-4,-6,-8,-10,-12,-14},
    grid=both,
    legend pos=north east,
    thick,
    title={$(T,\alpha, \gamma) = (30,0.001,-0.5)$}
]
\addplot[mark=o, blue] coordinates {
    (10, 0.9802554530941032)
    (20, 0.8286049752321615)
    (50, 0.3773840497283715)
    (100, 0.11251467131309723)
    (150, 0.0145298363665835)
    (200, 0.004216259455460214)
    (400, 1.943368372726511e-05)
};
\addplot[mark=square*, red!70!black] coordinates {
    (10, 0.5391080260343237)
    (20, 0.12628053278833976)
    (50, 0.0005799841527966209)
    (100, 1.2874032088309154e-08)
    (150, 3.8627187159246203e-13)
    (200, 0.0)
    (400, 0.0)
};
\end{axis}

\end{tikzpicture}
}
\scalebox{0.8}{%
\begin{tikzpicture}

\begin{axis}[
    width=0.35\textwidth,
    height=0.3\textwidth,
    xlabel={CG iterations},
    xmin=10, xmax=400,
    xtick={10,100,200,400},
    ymode=log,
    ytickten={0,-2,-4,-6,-8,-10,-12,-14},
    grid=both,
    legend pos=north east,
    thick,
    title={$(T,\alpha, \gamma) = (30,10^{-5},-0.5)$}
]
\addplot[mark=o, blue] coordinates {
    (10, 0.9989087399732387)
    (20, 0.9884862136778999)
    (50, 0.920748606212723)
    (100, 0.8168979711920907)
    (150, 0.6563098080968072)
    (200, 0.5541764296006213)
    (400, 0.23036840042425993)
};
\addplot[mark=square*, red!70!black] coordinates {
    (10, 0.953556703521092)
    (20, 0.8315099962292426)
    (50, 0.4269777668288865)
    (100, 0.07676504900653772)
    (150, 0.01858617161475176)
    (200, 0.0043076249959489755)
    (400, 1.2510527670168364e-05)
};
\end{axis}

\end{tikzpicture}
}
\caption{Gradient method for parabolic problem: Relative error over CG iterations for varying Tikhonov parameters, reaction terms and time horizons.}
\label{fig:grad:parabolic}
\end{figure}

In Figure~\ref{fig:ppcg:parabolic} we report the iteration numbers for the PPCG method for the fixed regularization parameter $\alpha = 10^{-3}$ here. However, we see again that while the feedbacked version is very robust with respect to increasing the horizon length, the original problem faces severe numerical difficulties in particular for the unstable case $\gamma = -0.5$ depicted in the lower row.

\begin{figure}[htb]
    \centering
\scalebox{0.8}{%
\begin{tikzpicture}

\begin{axis}[
    width=0.35\textwidth,
    height=0.3\textwidth,
    xlabel={CG iterations},
    ylabel={Relative error},
    xmin=5, xmax=100,
    xtick={10,30,50,70,100},
    ymode=log,
    ytickten={0,-2,-4,-6,-8,-10,-12,-14},
    grid=both,
    legend pos=north east,
    thick,
    title={$(T, \gamma) = (10,0.5)$}
]
\addplot[mark=o, blue] coordinates {
    (5, 0.8001613046811151)
    (10, 0.334008754687629)
    (20, 0.011958523245873482)
    (30, 0.0010760195330775974)
    (40, 6.0317063738765126e-05)
    (50, 9.632260172461268e-07)
    (70, 2.483714691213275e-10)
    (100, 6.43443708602032e-15)
};
\addlegendentry{Original}
\addplot[mark=square*, red!70!black] coordinates {
    (5, 0.7703323526778391)
    (10, 0.30280483455411195)
    (20, 0.010595624956097949)
    (30, 0.00013193564489260564)
    (40, 1.862209031192652e-05)
    (50, 1.733320281932484e-07)
    (70, 4.445046728536667e-12)
    (100, 0.0)
};
\addlegendentry{Feedbacked}
\end{axis}

\end{tikzpicture}
}
\scalebox{0.8}{%
\begin{tikzpicture}

\begin{axis}[
    width=0.35\textwidth,
    height=0.3\textwidth,
    xlabel={CG iterations},
    xmin=5, xmax=100,
    xtick={10,30,50,70,100},
    ymode=log,
    ytickten={0,-2,-4,-6,-8,-10,-12,-14},
    grid=both,
    legend pos=north east,
    thick,
    title={$(T, \gamma) = (30,0.5)$}
]
\addplot[mark=o, blue] coordinates {
    (5, 0.8567827784869313)
    (10, 0.7779716427573538)
    (20, 0.239195367856418)
    (30, 0.08064450748133314)
    (40, 0.01384393627757902)
    (50, 0.003327980919736818)
    (70, 0.0001098299894949521)
    (100, 9.050360338771021e-07)
};
\addplot[mark=square*, red!70!black] coordinates {
    (5, 0.7774062917344369)
    (10, 0.5180695953509853)
    (20, 0.12791974440407175)
    (30, 0.0123800336051935)
    (40, 0.002932166469481069)
    (50, 0.000764787341564989)
    (70, 1.0962200150202846e-05)
    (100, 1.8487354737178365e-08)
};
\end{axis}

\end{tikzpicture}
}
\scalebox{0.8}{%
\begin{tikzpicture}

\begin{axis}[
    width=0.35\textwidth,
    height=0.3\textwidth,
    xlabel={CG iterations},
    xmin=5, xmax=100,
    xtick={10,30,50,70,100},
    ymode=log,
    ytickten={0,-2,-4,-6,-8,-10,-12,-14},
    grid=both,
    legend pos=north east,
    thick,
    title={$(T,\gamma) = (50,0.5)$}
]
\addplot[mark=o, blue] coordinates {
    (5, 0.8752619652898339)
    (10, 0.7823320151536354)
    (20, 0.3175891630951712)
    (30, 0.12515734706283974)
    (40, 0.06395796658386114)
    (50, 0.012722478465550543)
    (70, 0.001869615554663436)
    (100, 6.397775209645658e-05)
};
\addplot[mark=square*, red!70!black] coordinates {
    (5, 0.7772406971596546)
    (10, 0.5557916930398292)
    (20, 0.14481341422679953)
    (30, 0.05392831574668583)
    (40, 0.009332051506731557)
    (50, 0.001753821274181866)
    (70, 8.197220252472123e-05)
    (100, 8.24252345299534e-07)
};
\end{axis}

\end{tikzpicture}
}\\
\scalebox{0.8}{%
\begin{tikzpicture}

\begin{axis}[
    width=0.35\textwidth,
    height=0.3\textwidth,
    xlabel={CG iterations},
    ylabel={Relative error},
    xmin=5, xmax=100,
    xtick={10,30,50,70,100},
    ymode=log,
    ytickten={0,-2,-4,-6,-8,-10,-12,-14},
    grid=both,
    legend pos=north east,
    thick,
    title={$(T, \gamma) = (10,0)$}
]
\addplot[mark=o, blue] coordinates {
    (5, 0.876710133995582)
    (10, 0.6670169426284571)
    (20, 0.03933899433040463)
    (30, 0.0029995147644178047)
    (40, 0.0004172769031511792)
    (50, 2.199290078204292e-05)
    (70, 6.619568939593204e-08)
    (100, 2.9838915475070055e-12)
};
\addplot[mark=square*, red!70!black] coordinates {
    (5, 0.7790297762565654)
    (10, 0.30851356200726626)
    (20, 0.010998296647076982)
    (30, 0.00018526436981143663)
    (40, 1.6377618312083584e-05)
    (50, 3.1126670650672986e-07)
    (70, 3.2281940379078534e-11)
    (100, 0.0)
};
\end{axis}

\end{tikzpicture}
}
\scalebox{0.8}{%
\begin{tikzpicture}

\begin{axis}[
    width=0.35\textwidth,
    height=0.3\textwidth,
    xlabel={CG iterations},
    xmin=5, xmax=100,
    xtick={10,30,50,70,100},
    ymode=log,
    ytickten={0,-2,-4,-6,-8,-10,-12,-14},
    grid=both,
    legend pos=north east,
    thick,
    title={$(T, \gamma) = (30,0)$}
]
\addplot[mark=o, blue] coordinates {
    (5, 0.9962550213714911)
    (10, 0.8326168429101802)
    (20, 0.4578159357895489)
    (30, 0.23452719689443186)
    (40, 0.10343343374840686)
    (50, 0.06689452414603396)
    (70, 0.0076117548168608915)
    (100, 0.00014330375462923874)
};
\addplot[mark=square*, red!70!black] coordinates {
    (5, 0.785816371495341)
    (10, 0.5270632972189708)
    (20, 0.1327521546079677)
    (30, 0.013183125305473551)
    (40, 0.003305105880394072)
    (50, 0.0008657475927028164)
    (70, 1.2372771977792146e-05)
    (100, 3.566896868458902e-08)
};
\end{axis}

\end{tikzpicture}
}
\scalebox{0.8}{%
\begin{tikzpicture}

\begin{axis}[
    width=0.35\textwidth,
    height=0.3\textwidth,
    xlabel={CG iterations},
    xmin=5, xmax=100,
    xtick={10,30,50,70,100},
    ymode=log,
    ytickten={0,-2,-4,-6,-8,-10,-12,-14},
    grid=both,
    legend pos=north east,
    thick,
    title={$(T, \gamma) = (50,0)$}
]
\addplot[mark=o, blue] coordinates {
    (5, 0.9994636968963782)
    (10, 0.9436258977443107)
    (20, 0.7992280227444969)
    (30, 0.38156533979093904)
    (40, 0.2844176569387488)
    (50, 0.17761337765600665)
    (70, 0.07941355898770563)
    (100, 0.008674945895127299)
};
\addplot[mark=square*, red!70!black] coordinates {
    (5, 0.7857190473670876)
    (10, 0.565149793231868)
    (20, 0.1495330885707933)
    (30, 0.05780060735553804)
    (40, 0.00992764164163244)
    (50, 0.002017174031099142)
    (70, 9.052921667614845e-05)
    (100, 1.3265053847723928e-06)
};
\end{axis}

\end{tikzpicture}
}
\scalebox{0.8}{%
\begin{tikzpicture}

\begin{axis}[
    width=0.35\textwidth,
    height=0.3\textwidth,
    xlabel={CG iterations},
    ylabel={Relative error},
    xmin=5, xmax=100,
    xtick={10,30,50,70,100},
    ymode=log,
    ytickten={0,-2,-4,-6,-8,-10,-12,-14},
    grid=both,
    legend pos=north east,
    thick,
    title={$(T, \gamma) = (10,-0.5)$}
]
\addplot[mark=o, blue] coordinates {
    (5, 0.9455222511525668)
    (10, 0.8209475368259452)
    (20, 0.13036206729883767)
    (30, 0.01671072386304267)
    (40, 0.007202999730300138)
    (50, 0.0016264052829117978)
    (70, 3.0470790163731014e-05)
    (100, 3.028188824922174e-08)
};
\addplot[mark=square*, red!70!black] coordinates {
    (5, 0.7884306957847413)
    (10, 0.315166960875018)
    (20, 0.011866087138935166)
    (30, 0.00024371343790657792)
    (40, 3.1492411169813396e-05)
    (50, 1.0763369210211925e-06)
    (70, 1.2342611527463634e-11)
    (100, 0.0)
};
\end{axis}

\end{tikzpicture}
}
\scalebox{0.8}{%
\begin{tikzpicture}

\begin{axis}[
    width=0.35\textwidth,
    height=0.3\textwidth,
    xlabel={CG iterations},
    xmin=5, xmax=100,
    xtick={10,30,50,70,100},
    ymode=log,
    ytickten={0,-2,-4,-6,-8,-10,-12,-14},
    grid=both,
    legend pos=north east,
    thick,
    title={$(T, \gamma) = (30,-0.5)$}
]
\addplot[mark=o, blue] coordinates {
    (5, 0.9964334488619728)
    (10, 0.9795468994142232)
    (20, 0.8282204976207466)
    (30, 0.8190738591961864)
    (40, 0.7050138725922295)
    (50, 0.39294989537316755)
    (70, 0.3327790684391965)
    (100, 0.11823589752720771)
};
\addplot[mark=square*, red!70!black] coordinates {
    (5, 0.7949054049347161)
    (10, 0.5370222917662403)
    (20, 0.13843220293316724)
    (30, 0.014100116922380137)
    (40, 0.00377682369822198)
    (50, 0.001002497703291598)
    (70, 1.8563734138905803e-05)
    (100, 4.094562835887798e-08)
};
\end{axis}

\end{tikzpicture}
}    
\scalebox{0.8}{%
\begin{tikzpicture}

\begin{axis}[
    width=0.35\textwidth,
    height=0.3\textwidth,
    xlabel={CG iterations},
    xmin=5, xmax=100,
    xtick={10,30,50,70,100},
    ymode=log,
    ytickten={0,-2,-4,-6,-8,-10,-12,-14},
    grid=both,
    legend pos=north east,
    thick,
    title={$(T, \gamma) = (50,-0.5)$}
]
\addplot[mark=o, blue] coordinates {
    (5, 0.9995713035120756)
    (10, 0.9995712929177519)
    (20, 0.9953905894953937)
    (30, 0.995390592951449)
    (40, 0.9733121069420667)
    (50, 0.9733021753363749)
    (70, 0.9295989094586826)
    (100, 0.8984810405129386)
};
\addplot[mark=square*, red!70!black] coordinates {
    (5, 0.7948814185810744)
    (10, 0.5754407416104831)
    (20, 0.1551520811269918)
    (30, 0.062324726413066904)
    (40, 0.010565865682195941)
    (50, 0.0023274306220867416)
    (70, 0.00010917466373715135)
    (100, 1.3956089051444946e-05)
};
\end{axis}

\end{tikzpicture}
} 
    \caption{PPCG method for parabolic problem. Left to right: increasing time horizon $T$. Top to bottom: Exponentially stable, stable and unstable uncontrolled equation.}
    \label{fig:ppcg:parabolic}
\end{figure}
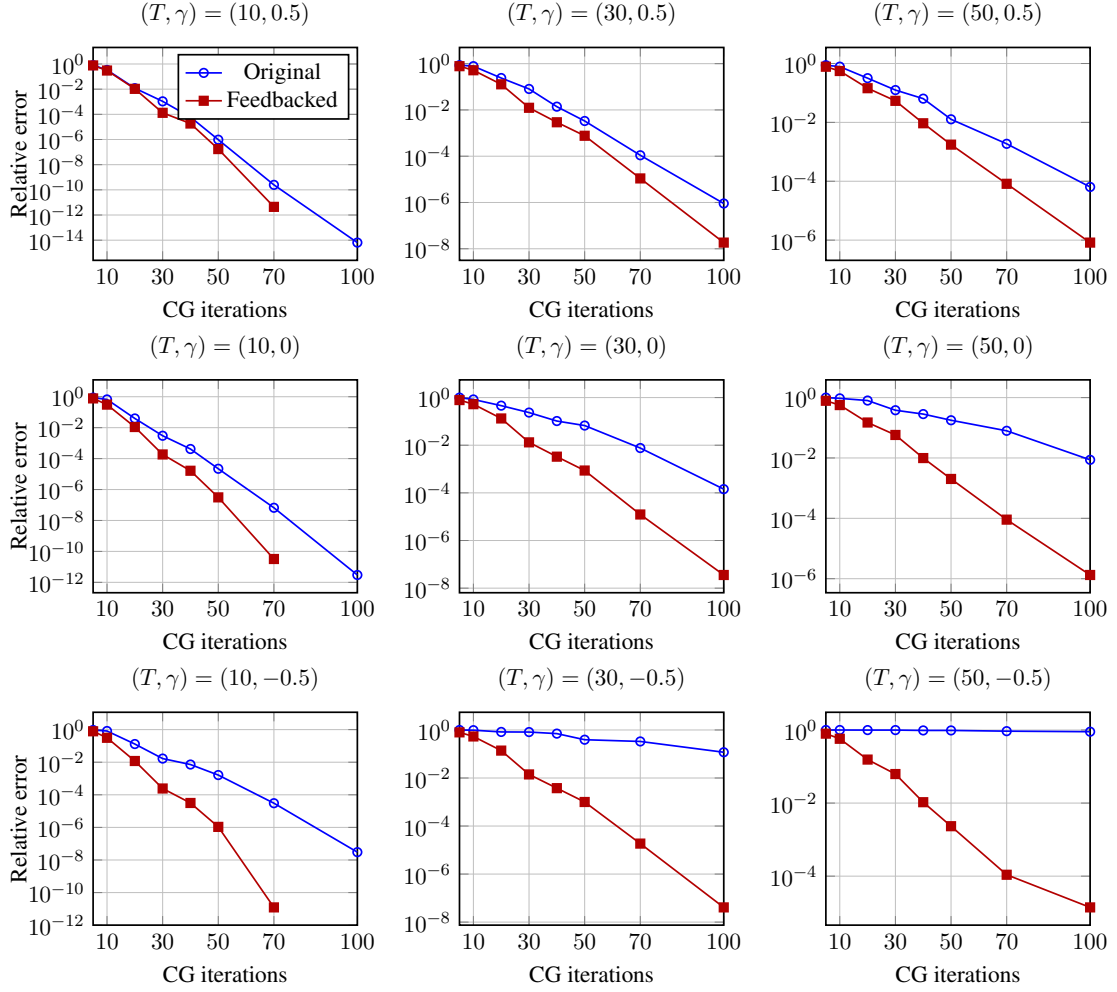

\medskip

\textbf{Optimization of the condition number \eqref{eq:condnumber}}.
We consider the heat equation with Neumann boundary conditions and with $\calC=\calB = I$. When choosing the feedback $\calK = -\delta I$ for $\delta > 0$ and by a similar argumentation as for the finite-dimensional case proven in Proposition~\ref{prop:ode} replacing the matrix exponential by the operator semigroup, one can derive the solution operator norm 
\begin{small}
  $$      \sigma_\calK \lesssim \frac{e^{-\delta T}-1}{-\delta} + 1.$$  
\end{small}
 If $\delta \to 0$ then the first term approaches $T$ (which may be seen using L'Hôpital's rule) that is typical for stable (but not asymptotically stable) systems due to integration of a constant function over $[0,T]$; a similar bound will be observed for the wave equation in the next subsection. For any $\delta > 0$, however, this term is bounded uniformly in $T$ and for $\delta \to \infty$ it approaches one. Again, we may plug this relation into \eqref{eq:condnumber} and optimize over $\delta$. This is illustrated in Figure~\ref{fig:condopt_p}.

    \begin{figure}[htb]
    \centering
	\includegraphics[width=.8\linewidth]{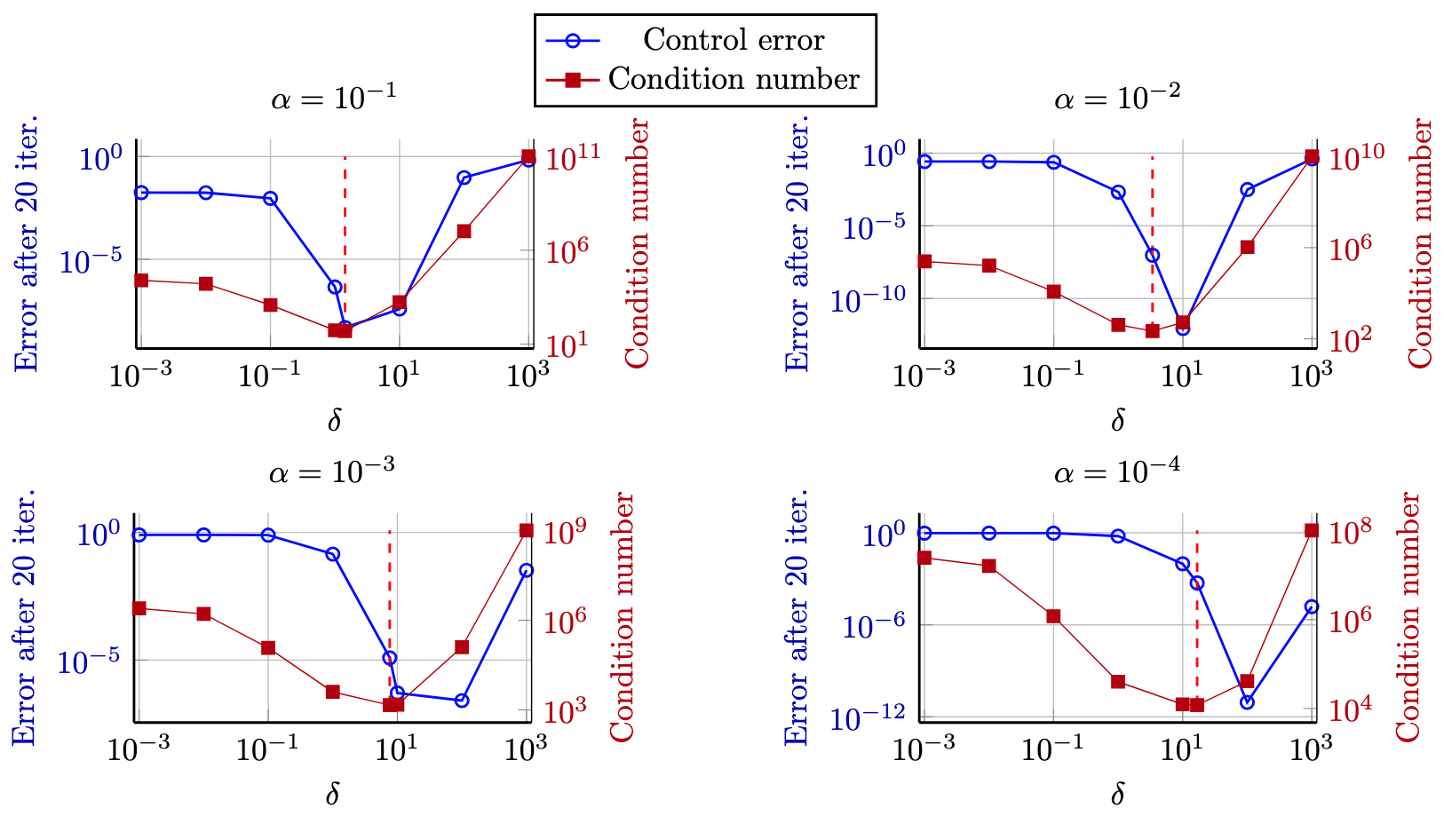}
    \caption{Error after a fixed number of iterations and condition number.}
    \label{fig:condopt_p}
\end{figure}

\subsection{Hyperbolic PDEs}\label{subsec:hyperbolic}
Let $\Omega \subset \R^n$ be a bounded open domain. Consider the controlled wave equation with homogeneous Dirichlet boundary conditions and prescribed initial data
\begin{equation}\label{eq:wave}
\begin{array}{rcll}  
     x_{tt}(t,\omega) - \Delta x(t,\omega)  &=& u(t, \omega) & \qquad \text{on } [0,T] \times  \Omega \\
     x(t,\omega) &=&0 & \qquad \text{on } [0,T] \times  \partial \Omega \\
     x(0,\omega) &=&x_0(\omega) & \qquad \text{on } \Omega \\
     x_t(0,\omega) &=&x_1(\omega) & \qquad \text{on } \Omega .
\end{array}
\end{equation}
Define the operator $A: L^2(\Omega) \supset H^2(\Omega) \cap H_0^1(\Omega) \to L^2(\Omega)$ with $A x = - \Delta x $.
Note that $A$ is self-adjoint and uniformly positive and there exists an operator square root $A^{1/2} : L^2(\Omega) \supset H_0^1(\Omega) \to L^2(\Omega)$ where $H_0^1(\Omega)$ is endowed with the \textit{energy inner product} $ \langle \cdot,\cdot \rangle_{H_0^1(\Omega)} =\langle A^{1/2} \cdot,A^{1/2}\cdot \rangle_{L^2(\Omega)}$. Let $H^{-1}(\Omega)$ be the dual space of $H_0^1(\Omega)$ with respect to the pivot space $L^2(\Omega)$. Precisely, $(H_0^1(\Omega), L^2(\Omega), H^{-1}(\Omega))$ defines a Gelfand triple and there exists a unique bounded extension of $A$ from $H_0^1(\Omega)$ to $H^{-1}(\Omega)$ that coincides with the Riesz isomorphism $\calR_{H_0^1(\Omega)}$, again denoted by $A: H_0^1(\Omega) \to H^{-1}(\Omega)$. For a Banach space $Z$, denote by $D'(0,T;Z)$ the space of $Z$-valued distributions on $(0,T)$, i.e., $D'(0,T;Z)\coloneqq L(C_c^\infty(0,T),Z)$. Since every continuous $H_0^1(\Omega)$-valued function defines an
$H_0^1(\Omega)$-valued distribution, the \textit{second distributional time derivative} $\ddot x\in D'(0,T;H_0^1(\Omega))$
is well-defined for every $x\in C([0,T];H_0^1(\Omega))$. Due to the continuous embedding $H_0^1(\Omega)\hookrightarrow H^{-1}(\Omega)$, we may equally regard $\ddot x\in D'(0,T;H^{-1}(\Omega)).$
Because $A\in L(H_0^1(\Omega),H^{-1}(\Omega))$, the pointwise mapping $t\mapsto Ax(t)$ belongs to $ C([0,T];H^{-1}(\Omega))\subset D'(0,T;H^{-1}(\Omega))$.
Consequently,
\begin{equation*}
    \tfrac{\mathrm d^2}{\mathrm dt^2}+A:
    C([0,T];H_0^1(\Omega)) 
    \rightarrow
    D'(0,T;H^{-1}(\Omega))
\end{equation*}
is well-defined.
Define the space and norm
\begin{align*}
    E(0,T;A) &\coloneqq \left\lbrace x \in C([0,T];H_0^1(\Omega)) \cap C^1([0,T];L^2(\Omega)) : \ddot x + A x \in L^2(0,T;L^2(\Omega)) \right\rbrace\\
     \| x \|_{E(0,T;A)} &\coloneqq  \| x \|_{C([0,T];H_0^1(\Omega))} + \| \dot x \|_{C([0,T];L^2(\Omega))} + \| \ddot x + A x\|_{L^2(0,T;L^2(\Omega))}.
\end{align*}
which forms a Banach space. Indeed, completeness follows by identifying $E(0,T;A)$ with the corresponding first-order energy space associated with the abstract wave operator $\begin{smallbmatrix}
    0 & I \\ -A & 0
\end{smallbmatrix}$.
Let $\calP \coloneqq L^2(0,T;L^2(\Omega)) \times H_0^1(\Omega) \times
L^2(\Omega)$ and define $\calA: E(0,T;A) \to \calP$ and $\calB: L^2(0,T;L^2(\Omega)) \to \calP$ by
\begin{equation}
    \calA x = \begin{smallpmatrix}
        \ddot x + A x \\ x(0) \\ \dot x(0)
    \end{smallpmatrix}, \qquad \calB u = \begin{smallpmatrix}
        u \vphantom{A} \\ 0 \vphantom{(0)} \\  0 \vphantom{(0)} 
    \end{smallpmatrix}.
\end{equation}
Introducing the first-order state variable $(x,v) \coloneqq (x, \dot x)$ (\textit{displacement-velocity} formulation) transforms the second-order equation into an abstract Cauchy problem generated by the associated wave operator. Consequently, $\calA$ is an isomorphism, see e.g.~\cite[Thm.~5.1.10]{curtain2020introduction}. We rewrite the hyperbolic control problem \eqref{eq:wave}
by 
\begin{equation}\label{eq:abstract_wave}
    \calA x - \calB u = \begin{smallbmatrix}
        0 \\ x_0 \\ x_1
    \end{smallbmatrix}.
\end{equation}
 For the cost functional, we introduce position and velocity observation $$\calC=\begin{smallbmatrix}
    I \\ \frac{\mathrm{d}}{\mathrm{d}t}
\end{smallbmatrix}: E(0,T;A) \to L^2(0,T;H_0^1(\Omega) \times L^2(\Omega)).$$

The following result shows a solution operator bound for the uncontrolled problem. In particular, we observe a dependence on the time horizon $T$ which (in constrast to the previously discussed parabolic case) is not exponential but linear; this is due to the energy conservation property of the wave equation and essentially stems from integrating a constant norm (energy) over the time horizon $T$. The proof is presented in Appendix \ref{subsec:app2}.
\begin{theorem}\label{thm:stab_hyperb}
    The operator $\calA$ is bounded, boundedly invertible and 
    \begin{equation}\label{eq:nonunif_bound_hyperb}
            \sigma_0=\|  \calC \blk \calA^{-1} \|_{\calP \to L^2(0,T;H_0^1(\Omega) \times L^2(\Omega))} \leq \max (\sqrt{2T}, T).
    \end{equation}
\end{theorem}
The previous theorem shows that $\calC\calA^{-1}$ is not uniformly bounded in the time horizon. We now show how to suitably design a feedback stabilization $\tcalA = \calA - \calB \calK$ such that $\calC\tcalA^{-1}$ is uniformly bounded in the time horizon. As the observation includes both position and velocity, in view of  Assumption~\ref{ass:feedback}\ref{ass:smallness} we may use the feedback operator
\begin{equation*}
   \calK : E(0,T;A) \to  L^2(0,T;L^2(\Omega)), \qquad \calK x = -(x + \dot x)
\end{equation*}
which induces linear friction in the model. The resulting closed loop system reads  $\ddot x + \dot x - \Delta x + x =0$ and in \textit{displacement-velocity} formulation,
\begin{equation*}
    \begin{bmatrix}
    \dot x \\ \dot v
\end{bmatrix} = \begin{bmatrix}
        0 & 1 \\ \Delta -1& - 1
    \end{bmatrix} \begin{bmatrix}
    x \\v
\end{bmatrix}.
\end{equation*}
The next result shows that using this feedback operator, the transformed system is exponentially stable and in particular, the feedbacked constraint operator is uniformly boundedly invertible in the time horizon. The proof is presented in Appendix \ref{subsec:app2}.
\begin{theorem}\label{thm:exp_stab_hyperb}
   The operator $ \widetilde \calA = \calA - \calB \calK$
fulfills
 \begin{equation}\label{eq:unif_bound_hyperb}
        \sigma_\calK=\| \calC \blk \tcalA^{-1} \|_{\calP \to L^2(0,T;H_0^1(\Omega) \times L^2(\Omega))} \leq C%
    \end{equation}
for a constant $C>0$ independent of the time-horizon $T$.
\end{theorem}
We illustrate the results for the gradient method in Figure~\ref{fig:grad:hyperbolic}. As the feedbacked system has a condition number bounded uniformly in the time horizon (due to Theorem~\ref{thm:exp_stab_hyperb}) and the original system has a bound depending on $T$ (see Theorem~\ref{thm:stab_hyperb}), we observe that when increasing the time horizon (left to right), the gradient method applied to the feedbacked problem more and more outperforms the gradient method for to the original problem.
\blk

\begin{figure}[htb]
    \centering
\scalebox{0.8}{%
\begin{tikzpicture}

\begin{axis}[
    width=0.35\textwidth,
    height=0.3\textwidth,
    xlabel={CG iterations},
    ylabel={Relative error},
    xmin=10, xmax=50,
    xtick={10,20,50,100,200},
    ymode=log,
    ytickten={0,-2,-4,-6,-8,-10,-12,-14},
    grid=both,
    legend pos=north east,
    thick,
    title={$(T,\alpha) = (10,0.001)$}
]
\addplot[mark=o, blue] coordinates {
    (10, 0.01751287649601164)
    (20, 9.477111129658123e-06)
        (50, 9.477111129658123e-14)
};
\addlegendentry{Original}
\addplot[mark=square*, red!70!black] coordinates {
    (10, 0.007672586577570736)
    (20, 2.907551033375415e-06)
        (50, 9.477111129658123e-14)
};
\addlegendentry{Feedbacked}
\end{axis}
\end{tikzpicture}
}
\scalebox{0.8}{%
\begin{tikzpicture}

\begin{axis}[
    width=0.35\textwidth,
    height=0.3\textwidth,
    xlabel={CG iterations},
    xmin=10, xmax=100,
    xtick={10,50,100,200},
    ymode=log,
    ytickten={0,-2,-4,-6,-8,-10,-12,-14},
    grid=both,
    legend pos=north east,
    thick,
    title={$(T,\alpha) = (30,0.001)$}
]
\addplot[mark=o, blue] coordinates {
    (10, 0.19179225491192184)
    (20, 0.006335766393586188)
    (50, 2.1462752462181757e-06)
    (100, 8.926437223200016e-13)
};
\addplot[mark=square*, red!70!black] coordinates {
    (10, 0.05941597848530918)
    (20, 0.0012641829982221638)
    (50, 2.673853185954666e-08)
    (100, 2.0609154292194055e-13)
};
\end{axis}

\end{tikzpicture}
}
\scalebox{0.8}{%
\begin{tikzpicture}

\begin{axis}[
    width=0.35\textwidth,
    height=0.3\textwidth,
    xlabel={CG iterations},
    xmin=10, xmax=100,
    xtick={10,50,100,200},
    ymode=log,
    ytickten={0,-2,-4,-6,-8,-10,-12,-14},
    grid=both,
    legend pos=north east,
    thick,
    title={$(T,\alpha) = (50,0.001)$}
]
\addplot[mark=o, blue] coordinates {
    (10, 0.41594123144699396)
    (20, 0.051259503201814195)
    (50, 0.00031299367406420757)
    (100, 2.286858942586622e-08)
};
\addplot[mark=square*, red!70!black] coordinates {
    (10, 0.09815129837154021)
    (20, 0.008146108610215712)
    (50, 1.6634035444967888e-06)
    (100, 2.0285751383548296e-12)
};
\end{axis}

\end{tikzpicture}
}\\
\scalebox{0.8}{%
\begin{tikzpicture}

\begin{axis}[
    width=0.35\textwidth,
    height=0.3\textwidth,
    xlabel={CG iterations},
    ylabel={Relative error},
    xmin=10, xmax=200,
    xtick={10,50,100,200},
    ymode=log,
    ytickten={0,-2,-4,-6,-8,-10,-12,-14},
    grid=both,
    legend pos=north east,
    thick,
    title={$(T,\alpha) = (10,10^{-5})$}
]
\addplot[mark=o, blue] coordinates {
    (10, 0.6606344295840225)
    (20, 0.23851233397327157)
    (50, 0.01505220619328308)
    (100, 8.088854116884556e-05)
    (200, 4.616527754370169e-09)
};
\addplot[mark=square*, red!70!black] coordinates {
    (10, 0.5779625783515272)
    (20, 0.20750711355316964)
    (50, 0.0055247715401029185)
    (100, 2.202974583132827e-05)
    (200, 4.119937050485917e-10)
};
\end{axis}

\end{tikzpicture}
}
\scalebox{0.8}{%
\begin{tikzpicture}

\begin{axis}[
    width=0.35\textwidth,
    height=0.3\textwidth,
    xlabel={CG iterations},
    xmin=10, xmax=200,
    xtick={10,50,100,200},
    ymode=log,
    ytickten={0,-2,-4,-6,-8,-10,-12,-14},
    grid=both,
    legend pos=north east,
    thick,
    title={$(T,\alpha) = (30,10^{-5})$}
]
\addplot[mark=o, blue] coordinates {
    (10, 0.8708962902086175)
    (20, 0.5484386136728564)
    (50, 0.19522635224976417)
    (100, 0.03318957645331909)
    (200, 0.0011411630179903466)
};
\addplot[mark=square*, red!70!black] coordinates {
    (10, 0.7594706224715455)
    (20, 0.43285882606779275)
    (50, 0.11930082376904741)
    (100, 0.01223770875237264)
    (200, 0.00010703739753784479)
};
\end{axis}

\end{tikzpicture}
}
\scalebox{0.8}{%
\begin{tikzpicture}

\begin{axis}[
    width=0.35\textwidth,
    height=0.3\textwidth,
    xlabel={CG iterations},
    xmin=10, xmax=200,
    xtick={10,50,100,200},
    ymode=log,
    ytickten={0,-2,-4,-6,-8,-10,-12,-14},
    grid=both,
    legend pos=north east,
    thick,
    title={$(T,\alpha) = (50,10^{-5})$}
]
\addplot[mark=o, blue] coordinates {
    (10, 0.9358957976889742)
    (20, 0.7450075952725995)
    (50, 0.39316701989374675)
    (100, 0.11787877169081454)
    (200, 0.016509783845688457)
};
\addplot[mark=square*, red!70!black] coordinates {
    (10, 0.8086783704646525)
    (20, 0.5874916599975385)
    (50, 0.1898510101126374)
    (100, 0.037881820964390245)
    (200, 0.001293915246537536)
};
\end{axis}

\end{tikzpicture}
}
\caption{Gradient method for hyperbolic problem: Relative error over CG iterations for varying regularization parameters and reaction terms.}
\label{fig:grad:hyperbolic}
\end{figure}
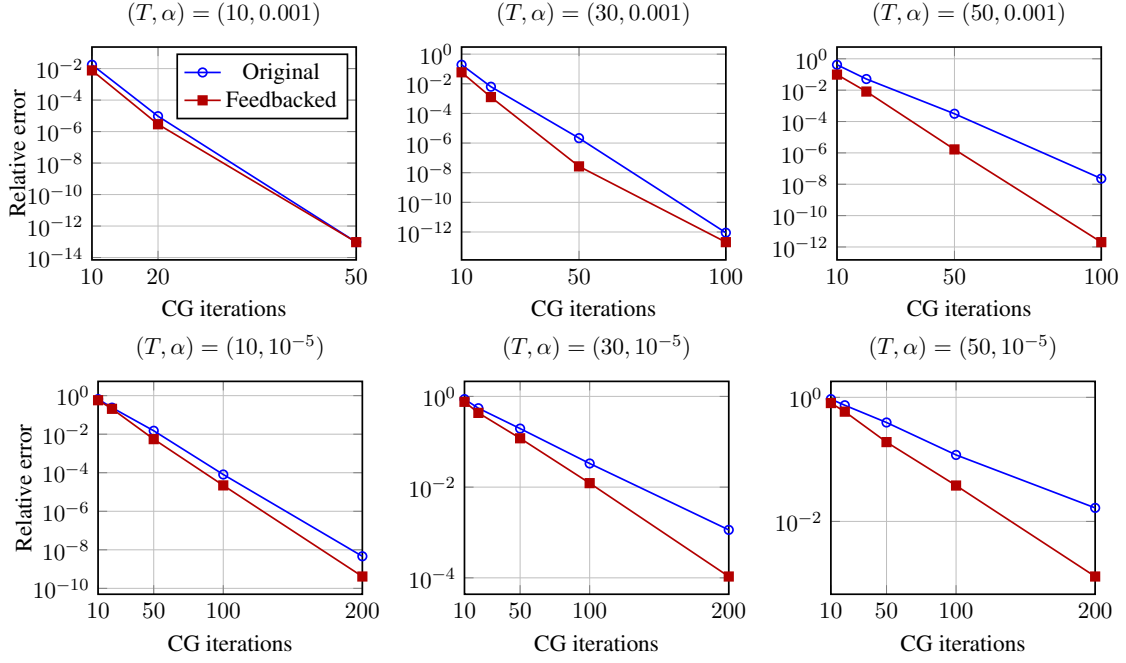

The results for the PPCG method are shown in Figure~\ref{fig:ppcg:wave} for $\alpha = 10^{-3}$. Here, we observe a similar behavior, that is, while both formulations perform equally well for small time horizons, the difference becomes more and more significant as we increase the time horizon.

\begin{figure}[htb]
    \centering
\scalebox{0.8}{%
\begin{tikzpicture}

\begin{axis}[
    width=0.35\textwidth,
    height=0.3\textwidth,
    xlabel={CG iterations},
    ylabel={Relative error},
    xmin=5, xmax=70,
    xtick={10,30,50,70,100},
    ymode=log,
    ytickten={0,-2,-4,-6,-8,-10,-12,-14},
    grid=both,
    legend pos=north east,
    thick,
    title={$T=10$}
]
\addplot[mark=o, blue] coordinates {
    (5, 0.471144196408362)
    (10, 0.0674079136723612)
    (20, 0.0007844559716823962)
    (30, 1.0657345867403937e-05)
    (40, 8.833919642288362e-08)
    (50, 1.0135818484640496e-09)
    (70, 6.837988875367045e-14)
    (100, 4.117911027335265e-14)
};
\addlegendentry{Original}
\addplot[mark=square*, red!70!black] coordinates {
    (5, 0.31918309023507446)
    (10, 0.04496632741904206)
    (20, 0.0003720763798118311)
    (30, 2.0794734388447563e-06)
    (40, 3.9242258607250266e-08)
    (50, 2.5904457193835756e-10)
    (70, 5.215714583056894e-14)
    (100, 4.109682618784825e-14)
};
\addlegendentry{Feedbacked}
\end{axis}

\end{tikzpicture}
}
\scalebox{0.8}{%
\begin{tikzpicture}

\begin{axis}[
    width=0.35\textwidth,
    height=0.3\textwidth,
    xlabel={CG iterations},
    xmin=5, xmax=100,
    xtick={10,30,50,70,100},
    ymode=log,
    ytickten={0,-2,-4,-6,-8,-10,-12,-14},
    grid=both,
    legend pos=north east,
    thick,
    title={$T=20$}
]
\addplot[mark=o, blue] coordinates {
    (5, 0.5729797569001548)
    (10, 0.2942738888196742)
    (20, 0.015119497810332857)
    (30, 0.00189936649609382)
    (40, 0.00016559600131420244)
    (50, 1.866315976621584e-05)
    (70, 5.481703290313505e-08)
    (100, 2.1180580793043703e-11)
};
\addplot[mark=square*, red!70!black] coordinates {
    (5, 0.34834914939732714)
    (10, 0.05915762157442206)
    (20, 0.003322940426067059)
    (30, 0.00014985436641517657)
    (40, 7.112748374156141e-06)
    (50, 6.392093232764277e-07)
    (70, 6.644103794245564e-10)
    (100, 9.005387159202798e-14)
};
\end{axis}

\end{tikzpicture}
}
\scalebox{0.8}{%
\begin{tikzpicture}

\begin{axis}[
    width=0.35\textwidth,
    height=0.3\textwidth,
    xlabel={CG iterations},
    xmin=5, xmax=100,
    xtick={10,30,50,70,100},
    ymode=log,
    ytickten={0,-2,-4,-6,-8,-10,-12,-14},
    grid=both,
    legend pos=north east,
    thick,
    title={$T=30$}
]
\addplot[mark=o, blue] coordinates {
    (5, 0.644608592055743)
    (10, 0.4403882699217282)
    (20, 0.04174084334918629)
    (30, 0.010825459629634578)
    (40, 0.002674450508490017)
    (50, 0.0005727877627981199)
    (70, 1.174563850758333e-05)
    (100, 6.216174042320608e-08)
};
\addplot[mark=square*, red!70!black] coordinates {
    (5, 0.38243627621926385)
    (10, 0.06616266506954513)
    (20, 0.0072967898516763975)
    (30, 0.0005148502550198909)
    (40, 4.079532805108566e-05)
    (50, 3.787640651565663e-06)
    (70, 2.0232500149781223e-08)
    (100, 1.4214044542977093e-11)
};
\end{axis}

\end{tikzpicture}
}
\scalebox{0.8}{%
\begin{tikzpicture}

\begin{axis}[
    width=0.35\textwidth,
    height=0.3\textwidth,
    xlabel={CG iterations},
    ylabel={Relative error},
    xmin=5, xmax=100,
    xtick={10,30,50,70,100},
    ymode=log,
    ytickten={0,-2,-4,-6,-8,-10,-12,-14},
    grid=both,
    legend pos=north east,
    thick,
    title={$T = 40$}
]
\addplot[mark=o, blue] coordinates {
    (5, 0.7233214322811498)
    (10, 0.46343415405795013)
    (20, 0.11359970835551332)
    (30, 0.021643844094331417)
    (40, 0.008460147300783163)
    (50, 0.0026601395416901395)
    (70, 0.00016416335573618586)
    (100, 2.3862182700344523e-06)
};
\addplot[mark=square*, red!70!black] coordinates {
    (5, 0.39319084407354893)
    (10, 0.0789630512648929)
    (20, 0.009069307755050389)
    (30, 0.0010619157247466439)
    (40, 8.111755649120188e-05)
    (50, 1.3248166691609693e-05)
    (70, 1.5636280758097447e-07)
    (100, 2.22646260037352e-10)
};
\end{axis}

\end{tikzpicture}
}
\scalebox{0.8}{%
\begin{tikzpicture}

\begin{axis}[
    width=0.35\textwidth,
    height=0.3\textwidth,
    xlabel={CG iterations},
    xmin=5, xmax=100,
    xtick={10,30,50,70,100},
    ymode=log,
    ytickten={0,-2,-4,-6,-8,-10,-12,-14},
    grid=both,
    legend pos=north east,
    thick,
    title={$T=50$}
]
\addplot[mark=o, blue] coordinates {
    (5, 0.6948124208693326)
    (10, 0.46266920379974535)
    (20, 0.23795308959141517)
    (30, 0.02912696681086819)
    (40, 0.014429448303085604)
    (50, 0.0053772673676302845)
    (70, 0.0006680314243765)
    (100, 1.939757326766026e-05)
};
\addplot[mark=square*, red!70!black] coordinates {
    (5, 0.3965342315572299)
    (10, 0.08048055639922759)
    (20, 0.009069689010035933)
    (30, 0.0014628224177524172)
    (40, 0.0001447487151987071)
    (50, 2.2687018771083166e-05)
    (70, 4.1129797404699084e-07)
    (100, 8.738023919440107e-10)
};
\end{axis}

\end{tikzpicture}
}
\caption{PPCG method for hyperbolic problem: Iterations for varying time horizons and regularization parameters.}
\label{fig:ppcg:wave}
\end{figure}
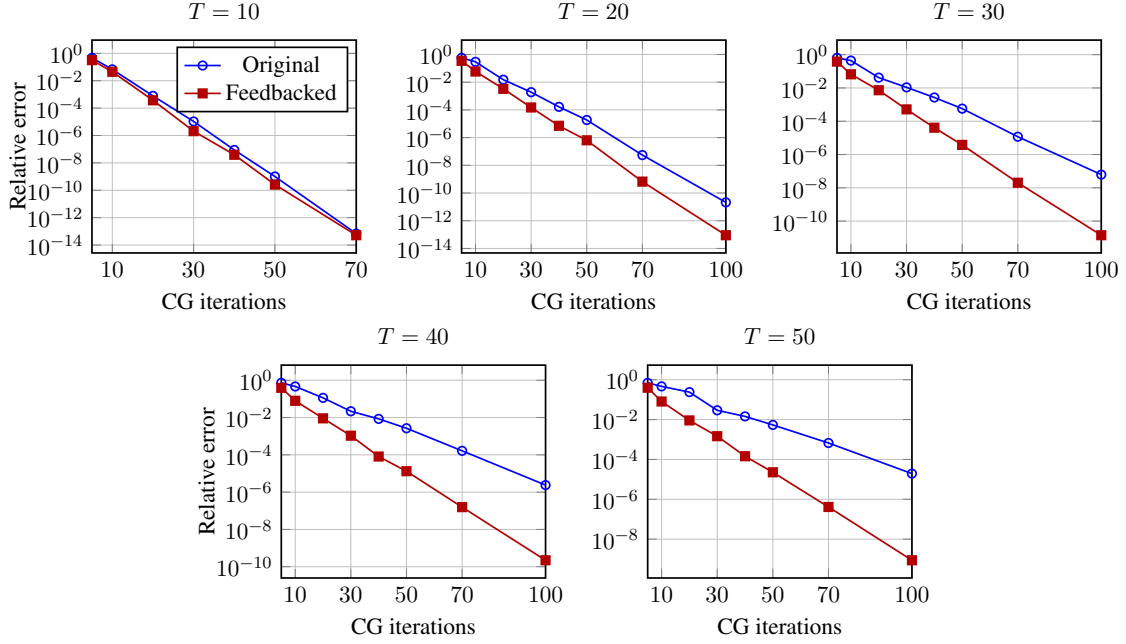

\bibliographystyle{plain}
\bibliography{references}
\appendix 
\begin{small}
\section{Supplementary proofs}
\subsection{Proof of \eqref{eq:toeplitz}}\label{subsec:toeplitz}
We first consider the case $a\neq 1$. To obtain a lower bound, we apply the Toeplitz matrix $T_N$ to the first canonical basis vector, which gives the lower bound
$\|T_N\|_2 \geq \|T_N e_1\|_2 = \left(\sum_{k=0}^{N-1} a^{2k}\right)^{1/2}$.
From $\|T_N\|_2 \leq \sqrt{\|T_N\|_1\|T_N\|_\infty}$ and
$\|T_N\|_1 = \|T_N\|_\infty = \frac{1-a^N}{1-a}$ we get an upper bound for the spectral norm. Putting both together for $a\neq 1$ yields  $\sqrt{\tfrac{1-a^{2N}}{1-a^2}}
    \;\leq\; \|T_N\|_2 \;\leq\;
    \tfrac{1-a^N}{1-a}$ that is, the claim.
We now consider the case \(a=1\). For the upper bound, we may estimate $\|T_N\|_2 \leq \|T_N\|_F = \sqrt{N^2} = N$.
For the lower bound, we choose $
\mathbf 1:=(1,\ldots,1)^\top\in\mathbb R^N$ for which we have $
T_N\mathbf 1
=
(1,2,\ldots,N)^\top,
$
leads to
$
\|T_N\|_2
\ge
\tfrac{\|T_N\mathbf 1\|_2}{\|\mathbf 1\|_2}
=
\left(
\tfrac{\sum_{k=1}^N k^2}{N}
\right)^{1/2}
=
\left(
\tfrac{(N+1)(2N+1)}{6}
\right)^{1/2}.$
As a result $\|T_N\|_2 \ge \frac{N}{\sqrt3}$ which together with the upper bound derived above yields the result.%
\subsection{Proof of Proposition~\ref{prop:ode}}\label{subsec:app1}
    Let $\gamma_0 : H^1(0,T; \R^n) \to \R^n$ be the boundary evaluation map, i.e., $\gamma_0(x)  = x(0)$. It is well-known that $\gamma_0$ is a bounded operator, see \cite[Thm 7.8.1]{aubin2000applied}. Hence, it is enough to prove boundedness of $ \tfrac{\mathrm{d}}{\mathrm{d}t}-A:  H^1(0,T; \R^n) \to L^2(0,T; \R^n)$ which is immediate from
    \begin{equation*}
        \|  \tfrac{\mathrm{d}}{\mathrm{d}t}x-Ax \|_{L^2(0,T; \R^n)}^2 \leq  \|  \tfrac{\mathrm{d}}{\mathrm{d}t}x \|_{L^2(0,T; \R^n)}^2  +  \|  Ax \|_{L^2(0,T; \R^n)}^2 \leq \operatorname{max}(1, \|A\|^2) \| x \|_{H^1(0,T; \R^n)}^2
    \end{equation*}
    for all $x \in H^1(0,T; \R^n)$. We proceed to show bijectivity of $\calA$. To this end, let $x \in \ker \calA$. In view of formula \eqref{eq:var_of_const}, we obtain
$ 
        x(t) = e^{tA}x_0 = e^{tA}0 =0,
  $
    which proves $x=0$. Furthermore, let $(f, x_0) \in L^2(0,T; \R^n) \times \R^n$ and denote by $x \in H^1(0,T; \R^n)$ the unique mild solution of 
    \begin{equation*}
        \dot x(t) = A x(t) + f(t)  , \qquad x(0) = x_0.
    \end{equation*}
    Then, $\calA x = \begin{smallbmatrix}
        f \\ x_0
    \end{smallbmatrix}$ and hence, $\ran \calA =  L^2(0,T; \R^n) \times \R^n$. Hence, there exists $\calA^{-1}:  L^2(0,T; \R^n) \times \R^n \to  H^1(0,T; \R^n)$ given by 
    \begin{equation*}
        \calA^{-1} \begin{smallbmatrix}
        f \\ x_0
    \end{smallbmatrix} = e^{\cdot A}x_0 + \int_0^\cdot e^{(\cdot-s)A} f(s)\,\mathrm{d}s \eqqcolon  x(\cdot).
    \end{equation*}
    We remain to show boundedness of $\calA^{-1}$. First, assume that $\omega \neq 0$. Due to \eqref{eq:Gronwall} we observe
    \begin{align*}
           \int_0^{T} \left\|e^{t A}x_0 \right\|_{\R^n}^2 \, \mathrm{d}t  %
           \leq M^2 \left\|x_0 \right\|_{\R^n}^2 \int_0^{T} e^{2t\omega}  \, \mathrm{d}t   \leq  M^2\frac{e^{2T\omega}-1}{2\omega} \left\|x_0 \right\|_{\R^n}^2 .
    \end{align*}
    Moreover, tedious but straightforward computations show that 
\begin{align*}
\int_0^T \!\!\int_0^t \|e^{(t-s)A}f(s)\|^2 \,\mathrm ds\,\mathrm dt
&\le
M^2 \int_0^T \!\!\int_0^t
e^{2(t-s)\omega}\|f(s)\|^2
\,\mathrm ds\,\mathrm dt \\
&=
M^2 \int_0^T
\|f(s)\|^2 \left(\int_s^T e^{2(t-s)\omega}\,\mathrm dt\right) \,\mathrm ds \\
&=
M^2 \int_0^T
\tfrac{e^{2(T-s)\omega}-1}{2\omega}
\|f(s)\|^2 \,\mathrm ds \le
M^2 \tfrac{e^{2T\omega}-1}{2\omega}
\|f\|_{L^2(0,T;\mathbb R^n)}^2 .
\end{align*}
Combining these two estimates yields the $L^2$-estimate
    \begin{equation*}
        \| x \|_{L^2(0,T; \R^n)}^2 \leq M^2 \tfrac{e^{2T\omega} -1}{\omega} \left( \| x_0 \|_{\R^n}^2 +\| f\|_{L^2(0,T;\R^n)}^2 \right).
    \end{equation*}
 For the full $H^1$-estimate, we compute
    \begin{align*}
    \| \calA^{-1} \begin{smallbmatrix}
        f \\ x_0
    \end{smallbmatrix} \|_{H^1(0,T; \R^n)}^2 
    &= \| x \|_{L^2(0,T; \R^n)}^2 +  \| \dot x \|_{L^2(0,T; \R^n)}^2 \\
    &= \| x \|_{L^2(0,T; \R^n)}^2 +  \| A x + f\|_{L^2(0,T; \R^n)}^2 \\
    & \leq \left(1 + 2 \|A\|^2\right) \| x \|_{L^2(0,T; \R^n)}^2 + 2 \| f\|_{L^2(0,T; \R^n)}^2 \\
     &\leq \left( \tfrac{M^2}{\omega} \left(1 + 2 \|A\|^2\right)  (e^{2T\omega}-1)+2\right) \hspace{-1mm}\left(\left\|x_0 \right\|_{\R^n}^2  +  \| f \|_{L^2(0,T; \R^n)}^2\right) .
    \end{align*}
    The case $\omega = 0$ is derived analogously. 
\begin{flushright}$\square$\end{flushright}

\subsection{Proofs of Theorem \ref{thm:stab_hyperb} and Theorem \ref{thm:exp_stab_hyperb}}\label{subsec:app2}
In this proof, we abbreviate the Gelfand triple $(H_0^1(\Omega), L^2(\Omega), H^{-1}(\Omega))$ by $(H_{1/2}, H,H_{-1/2})$. The operator $\calA$ is bounded by construction of the solution space $E(0,T;A)$ and continuity of the evaluation operators $x \mapsto x(0)$ and $x \mapsto \dot{x}(0)$. Let $D \in L(H)$ fulfill $\langle Dv,v\rangle_H \geq 0$ for all $v \in H$. Let $G: X \coloneqq H_{1/2} \times H \supset \dom G \to X$ be defined by
\begin{equation*}
    \dom G = \left\lbrace (x,v) \in H_{1/2} \times H_{1/2} : Ax \in H \right\rbrace, \qquad G \begin{smallbmatrix}
        x \\ v 
    \end{smallbmatrix}  = \begin{smallbmatrix}
        v \\ -Ax - D v
    \end{smallbmatrix}
\end{equation*}
It is well-known, that is $G$ maximally dissipative and hence the generator of a contraction semigroup $(T(t))_{t \geq 0}$ on $X$. Observe that Theorem \ref{thm:stab_hyperb} corresponds to the choice $D=0$, whereas Theorem \ref{thm:exp_stab_hyperb} is obtained for $D=I$. For $(x(0),v(0))=(x_0,v_0) \in X$, the unique mild solution of the inhomogeneous Cauchy problem 
\begin{equation*}
   \begin{smallbmatrix}
        \dot x(t) \\ \dot v(t) 
    \end{smallbmatrix} = G \begin{smallbmatrix}
        x(t) \\ v(t) 
    \end{smallbmatrix} + \begin{smallbmatrix}
       f_1(t) \\ f_2(t)
    \end{smallbmatrix}, \ \text{is given by } \begin{smallbmatrix}
        x(t) \\ v(t) 
    \end{smallbmatrix} \coloneqq  T(t)\begin{smallbmatrix}
        x_0 \vphantom{(0)} \\ x_1 \vphantom{(0)}
    \end{smallbmatrix}  + \int_0^t T(t-s)\begin{smallbmatrix}
       f_1(s) \\ f_2(s)
    \end{smallbmatrix} \, \mathrm{d}s.
\end{equation*}
Since the above equation is the first order (in time) formulation of \eqref{eq:wave} (with $ f_1=0,f_2 =f$), we obtain 
\begin{equation}\label{eq:solution_hyperb}
    x(t) = \pi_1 T(t)\begin{smallbmatrix}
        x_0 \vphantom{(0)} \\ x_1 \vphantom{(0)}
    \end{smallbmatrix} + \int_0^t \pi_1 T(t-s)\begin{smallbmatrix}
       0 \\ f(s)
    \end{smallbmatrix}  \, \mathrm{d}s
\end{equation}
where $\pi_1 : X \to H_{1/2}$ denotes the projection onto the first component. Note that this map fulfills $\| \pi_1 \|_{X \to H_{1/2}}=1$. Let $(f,x_0,x_1) \in L^2(0,T;H) \times H_{1/2} \times H$ be arbitrary. A combination of the above and \cite[Prop. 9]{farkas2025dissipativity} yields that the unique solution $x \in E(0,T;A)$ of 
\begin{equation*}
    \calA x = \begin{smallbmatrix}
        f \\ x_0 \\ x_1
    \end{smallbmatrix}, \qquad \widetilde{\calA}x\coloneqq (\calA - \calB \calK)x=\begin{smallbmatrix}
        f \\ x_0 \\ x_1
    \end{smallbmatrix}
\end{equation*}
is given by \eqref{eq:solution_hyperb}, respectively. Hence, both $\calA$ and $\widetilde{\calA}=\calA - \calB \calK$ are invertible and straightforward computations show that the inverses are continuous. We now show the upper bounds on $  \| \calC \calA^{-1} \|,  \| \calC \tcalA^{-1} \|$, respectively. It is well-known, that if $D=0$, then $G$ is skew-adjoint and hence the generator of a unitary group $(T(t))_{t \geq 0}$ on $H_{1/2} \times H$. In this case,
\begin{align*}
         \left|\left|\calC \calA^{-1}\begin{smallbmatrix}
        f \\ x_0 \\ x_1
    \end{smallbmatrix}  \right|\right|_{L^2(0,T;H_{1/2} \times H)}^2 
        &= \int_0^T \|x(t)\|_{H_{1/2}}^2 +\|\dot x(t)\|_{H}^2 \, \mathrm{d}t \\
        &= \int_0^T \left\| T(t)\begin{smallbmatrix}
        x_0 \\ x_1 
    \end{smallbmatrix} + \int_0^t T(t-s)\begin{smallbmatrix}
       0 \\ f(s)
    \end{smallbmatrix}  \, \mathrm{d}s\right\|_{X}^2 \, \mathrm{d}t \\
        &\leq \int_0^T 2 \| \begin{smallbmatrix}
        x_0 \\ x_1 
    \end{smallbmatrix} \|_X^2 \, \mathrm{d}t + \int_0^T t \int_0^t 2 \| f(s) \|_H^2 \, \mathrm{d} s \, \mathrm{d}t  \\
     &\leq  \operatorname{max}(2T,T^2) \left( \| \begin{smallbmatrix}
        x_0 \\ x_1 
    \end{smallbmatrix} \|_X^2 + \| f  \|_{L^2(0,T;H)}^2 \right)
    \end{align*}
    and hence \eqref{eq:nonunif_bound_hyperb}. 
 Conversely, if $D=I$, then $G$ is the generator of an exponentially stable semigroup $(T(t))_{t \geq 0}$ on $H_{1/2} \times H$,  see e.g. \cite{batkai2004exponential,hryniv2003exponential}. Hence, there exist constants $M \geq 0$, $\omega>0$ such that $\|T(t)\|_{H_{1/2}\times H \to H_{1/2}\times H} \leq Me^{-\omega t}$. Therefore,
\begin{equation*}
    \left\| \int_0^t T(t-s) \begin{smallbmatrix}
        0 \\ f(s) \end{smallbmatrix}\,\mathrm{d}s \right\|_X \leq M \int_0^t e^{-\omega(t-s)} \|f(s)\|_H\,\mathrm{d}s.
\end{equation*}
Squaring and applying the weighted Cauchy-Schwarz inequality yields
\begin{align*}
    \left( \int_0^t e^{-\omega(t-s)} \|f(s)\|_H\,\mathrm{d}s \right)^2 
    &=
    \left(\int_0^te^{-\omega(t-s)/2}e^{-\omega(t-s)/2} \|f(s)\|_H \,\mathrm{d}s \right)^2\\
    &\leq \left( \int_0^t e^{-\omega(t-s)} \,\mathrm{d}s \right) \left(\int_0^t e^{-\omega(t-s)}\|f(s)\|_H^2 \,\mathrm{d}s\right).
\end{align*}
Since
$
    \int_0^t e^{-\omega(t-s)}\,\mathrm{d}s = \frac{1-e^{-\omega t}}{\omega}\leq \frac{1}{\omega},
$
it follows that
\begin{equation*}
    \left\| \int_0^t T(t-s) \begin{smallbmatrix}
        0 \\ f(s)
    \end{smallbmatrix} \,\mathrm{d}s \right\|_X^2 \leq \frac{M^2}{\omega} \int_0^t e^{-\omega(t-s)} \|f(s)\|_H^2 \,\mathrm{d}s.    
\end{equation*}
Consequently,
    \begin{align*}
         \left|\left|\calC \widetilde{\calA}^{-1}\begin{smallbmatrix}
        f \\ x_0 \\ x_1
    \end{smallbmatrix}  \right|\right|^2%
        &= \int_0^T \|x(t)\|_{H_{1/2}}^2 +\|\dot x(t)\|_{H}^2 \, \mathrm{d}t \\
        &= \int_0^T \left\| T(t)\begin{smallbmatrix}
        x_0 \\ x_1 
    \end{smallbmatrix} + \int_0^t T(t-s)\begin{smallbmatrix}
       0 \\ f(s)
    \end{smallbmatrix}  \, \mathrm{d}s\right\|_{X}^2 \, \mathrm{d}t \\
        &\leq \int_0^T 2M^2e^{-2\omega t} \| \begin{smallbmatrix}
        x_0 \\ x_1 
    \end{smallbmatrix} \|_X^2 \, \mathrm{d}t + \frac{2M^2}{\omega} \int_0^T 
\int_0^t
e^{-\omega(t-s)}
\|f(s)\|_H^2
\,\mathrm{d}s\, \mathrm{d}t  \\
    &\leq M^2 \tfrac{1-e^{-2\omega T}}{\omega}  \| \begin{smallbmatrix}
        x_0 \\ x_1 
    \end{smallbmatrix} \|_X^2 + 2M^2 \int_0^T \| f(s) \|_H^2 \left( \int_s^T e^{-\omega (t-s)}  \, \mathrm{d}t \right) \, \mathrm{d} s .
    \end{align*}
Since  $\int_s^T e^{-\omega (t-s)}\, \mathrm{d} t = \int_0^{T-s} e^{-\omega r}\, \mathrm{d} r 
    \leq \int_0^\infty e^{-\omega r}\, \mathrm{d} r = \frac{1}{\omega}$
and $T \mapsto \tfrac{1-e^{-2\omega T}}{\omega}$ is monotonically decreasing,\eqref{eq:unif_bound_hyperb} follows due to
\begin{equation*}
    \left|\left|\calC \widetilde{\calA}^{-1}\begin{smallbmatrix}
        f \\ x_0 \\ x_1
    \end{smallbmatrix}  \right|\right|_{L^2(0,T;H_{1/2} \times H)}^2 \leq \max \left(\tfrac{M^2}{\omega},\tfrac{2M^2}{\omega} \right) \left| \left| \begin{smallbmatrix}
        f \\ x_0 \\ x_1
    \end{smallbmatrix}\right| \right|_{L^2(0,T; H) \times H^{1/2} \times H}^2
\end{equation*}

\begin{flushright}$\square$\end{flushright}
\end{small}

\end{document}